\documentclass[10pt]{amsart}
\usepackage{latexsym}
\usepackage[colorlinks = true,
            linkcolor = blue,
            urlcolor  = blue,
            citecolor = blue,
            anchorcolor = blue]{hyperref}
\usepackage[all]{xy}
\usepackage{pdfpages}

\usepackage{amsmath}
\usepackage{xypic}

\usepackage{amsmath}
\usepackage{amssymb}
\usepackage{amsfonts}
\usepackage{amsthm} %% Take out for NMJ style
\usepackage{mathrsfs}
\usepackage[utf8]{inputenc}

\usepackage{fmtcount}
\usepackage{refcount}

\DeclareUnicodeCharacter{00A0}{ }

\DeclareFontFamily{U}{wncy}{}
\DeclareFontShape{U}{wncy}{m}{n}{<->wncyr10}{}
\DeclareSymbolFont{mcy}{U}{wncy}{m}{n}
\DeclareMathSymbol{\Sh}{\mathord}{mcy}{"58}

\renewcommand{\implies}{\Rightarrow}

\newtheorem{theo}{Theorem}
\newtheorem{prop}[theo]{Proposition}

\newtheorem{defi}[theo]{Definition}

\newtheorem{coro}[theo]{Corollary}
\newtheorem{lemm}[theo]{Lemma}

\newtheorem*{prob}{Problem}
\theoremstyle{definition}

\newtheorem*{conv}{Convention}
\newtheorem{exam}[theo]{Example}
\newtheorem{rema}[theo]{Remark}

\newcommand{\FF}{\mathbb{F}}

\newcommand{\NN}{\mathbb{N}}
\newcommand{\OO}{\mathcal{O}}

\newcommand{\p}{\mathfrak{p}}

\newcommand{\q}{\mathfrak{q}}

\newcommand{\cS}{\mathcal{S}}

\newcommand{\ZZ}{\mathbb{Z}}
\newcommand{\zpi}{\mathbb{Z}[\tfrac{1}{p}]}
\newcommand{\zll}{\mathbb{Z}_{(l)}}

\DeclareMathOperator{\Ab}{Ab}

\DeclareMathOperator{\Spec}{Spec}
\DeclareMathOperator{\uSpec}{\underline{Spec}}

\DeclareMathOperator{\Tr}{Tr}

\DeclareMathOperator{\id}{id}

\DeclareMathOperator{\Frac}{Frac}

\DeclareMathOperator{\ind}{``\underset{\longrightarrow}{lim}''}

\DeclareMathOperator{\cosk}{cosk}

\newcommand{\PreShv}{\mathsf{PreShv}}
\newcommand{\Shv}{\mathsf{Shv}}
\newcommand{\Sm}{\mathsf{Sm}}

\newcommand{\Sch}{\mathsf{Sch}}
\newcommand{\SCH}{\mathsf{SCH}}

\newcommand{\Nis}{\mathsf{Nis}}
\newcommand{\et}{\operatorname{et}}

\newcommand{\PN}{\mathsf{PseNor}}
\newcommand{\PIC}{\mathsf{PseIntClo}}
\newcommand{\ulh}{\underline{H}}

\newcommand{\Zll}{\ZZ_{(l)}}

\newcommand{\op}{{\operatorname{op}}}
\newcommand{\eff}{{\operatorname{eff}}}
\newcommand{\gm}{{\operatorname{gm}}}
\newcommand{\proper}{{\operatorname{prop}}}

\newcommand{\fps}{{\operatorname{fps}}}
\newcommand{\fpsl}{{\operatorname{fps\!}l'}}

\newcommand{\h}{{{\operatorname{h}}}}
\newcommand{\cdh}{{{\operatorname{cdh}}}}

\newcommand{\ldh}{{l{\operatorname{dh}}}}
\newcommand{\uh}{\operatorname{uh}}
\newcommand{\cduh}{\operatorname{cduh}}
\newcommand{\red}{{\operatorname{red}}}
\newcommand{\ic}{{\operatorname{ic}}}
\newcommand{\pic}{{\operatorname{pic}}}
\newcommand{\length}{{\operatorname{length}}}
\newcommand{\Bl}{{\operatorname{Bl}}}

\newcommand{\cdd}{{\operatorname{cdd}}}

\begin{document}

\title{A better comparison of $\cdh$- and $\ldh$-cohomologies}
\author{Shane Kelly}
%\author{
%  \name{Shane Kelly}
%  \authorfootnote{Author footnote}
%  \address{Department of Mathematics \\
%Tokyo Institute of Technology \\
%2-12-1 Ookayama, Meguro-ku \\
%Tokyo 152-8551, Japan}
%  \email{shanekelly@math.titech.ac.jp}
%} 

%\subjclass[2010]{14A05, 18F10, 14F42} % ADD IN FOR NMJ style
%% 14A05  Relevant commutative algebra (to algebraic geometry)
%% 14F20  	Étale and other Grothendieck topologies and (co)homologies
%% 18F10  	Grothendieck topologies
%% 14G17  Positive characteristic ground fields
%% 14F42  	Motivic cohomology; motivic homotopy theory

\begin{abstract}
In order to work with non-Nagata rings which are Nagata ``up-to-completely-decomposed-universal-homeomorphism'', specifically finite rank hensel valuation rings, we introduce the notions of \emph{pseudo-integral closure}, \emph{pseudo-normalisation}, and \emph{pseudo-hensel valuation ring}.

We use this notion to give a shorter and more direct proof that $H_\cdh^n(X, F_\cdh) = H_\ldh^n(X, F_\ldh)$ for homotopy sheaves $F$ of modules over the $\zll$-linear motivic Eilenberg-Maclane spectrum. This comparison is an alternative to the first half of the author's volume Astérisque 391 whose main theorem is a cdh-descent result for Voevodsky motives. 

The motivating new insight is really accepting that Voevodsky's motivic cohomology (with $\zpi$-coefficients) is invariant not just for nilpotent thickenings, but for all universal homeomorphisms.
\end{abstract}

%\tableofcontents

\maketitle

\begin{center}
Warmly dedicated to Shuji Saito on his 60th birthday.
\end{center}

%\vspace{0.5cm}

%\tableofcontents

\section{Introduction}

%The reader not interested in applications is invited to go directly to the ``Main result'' subsection on page~\pageref{sec:mainResult}. 

%\subsection*{Context---(compact support) motives of singular schemes} \label{sec:context}
\textbf{Context---motives of singular schemes (with compact support).} %
In \cite{Voe00}, Voevodsky constructed a triangulated category of motives $DM^\eff_\gm(k)$ using smooth schemes. In order to 
 \begin{itemize}
  \item extend the motive functor $M: \Sm_k {\to} DM^\eff_\gm(k)$ from smooth $k$-schemes to all separated finite type $k$-schemes, and 

  \item have access to a well-defined theory of motives with compact support $M^c: \Sch_k^{\proper} {\to} DM^\eff_\gm(k)$, %, cf.~\cite{KelCdh}.
\end{itemize}
he proves a $\cdh$-descent result. However, his proof only works in the presence of strong resolution of singularities (for example, in characteristic zero). In \cite{Kel17}, the resolution of singularities assumption was removed, at least if one works with $\ZZ[\tfrac{1}{p}]$-coefficients, where $p$ is the exponential characteristic of the base field. The proof in \cite{Kel17} has two main steps, namely \cite[Cor.2.5.4]{Kel17} and \cite[Thm.3.2.12]{Kel17}. This present article provides a shortcut to the first one. For more details about the strategy of \cite{Kel17} we recommend consulting \cite[Chap.1]{Kel17} and/or \cite[Chap.4]{Kel17}.

\textbf{Main result.} \label{sec:mainResult} The main theorem of this article is:

\begin{theo} {(cf{.} Thm.\ref{theo:mainTheorem})}
Suppose $S$ is a finite dimensional noetherian separated scheme of positive characteristic $p \neq l$, $\Sch_S$ the category of separated finite type $S$-schemes, and $F$ a presheaf of $\Zll$-modules on $\Sch_S$ satisfying \emph{all} of:
\begin{enumerate}
 \item[($\uh$-invariance)] $F(X) \cong F(Y)$ for any universal homeomorphism $Y \to X$. 

 \item[(Traces)] $F$ has covariant ``trace'' morphisms associated to finite flat surjective morphisms, see Definition~\ref{def:traces}.

 \item[(G1)] $F(R) \subseteq F(\Frac(R))$ for every finite rank hensel valuation ring\footnote{We implicitly extend $F$ to all quasi-compact separated $S$-schemes using left Kan extension. Cf.Conventions on page~\pageref{convention}.} $R$.

 \item[(G2)] $F(R) \to F(R/\p)$ is surjective for every finite rank hensel valuation ring $R$ and prime $\p \subset R$. 
\end{enumerate}
Then the canonical comparison morphism is an isomorphism:
\[ H_\cdh^n(S, F_\cdh) \stackrel{\sim}{\to} H_\ldh^n(S, F_\ldh). \]
\end{theo}

However the main \emph{result} of this article is its proof. The proof of \cite[Cor.2.5.4]{Kel17} is a poorly structured collection of lemmas, which are difficult to arrange into some kind global narrative, and the hypotheses of \cite[Cor.2.5.4]{Kel17} are an awkward list of very special properties that don’t give much insight into why the cohomologies should agree.

The proof of Theorem~\ref{theo:mainTheorem} on the other hand, is short%
\footnote{Even though the proof ``finishes'' on page~\pageref{sec:cdhldh}, of course this introductory section does not form part of the proof, most if not all of Sections~\ref{sec:uh} and \ref{sec:hvr} is background scheme theory included for the convenience of the reader, and most of Section~\ref{sec:pn} is routine checking that pseudo-integral closures have the properties that we want. If one was writing in the more concise style preferred by some authors, one could fit the proof in 10 pages, probably less.} %
 and linear, and one can mostly explain how its hypotheses are used in the proof: $\zll$-linearity and traces are to give descent for \mbox{finite-flat-surjective-prime-to-$l$} morphisms, cf.Lem.\ref{lemm:hcVanishIntro}, universal homeomorphism invariance is to correct non-Nagata-ness of hensel valuation rings, cf.Lem.\ref{lem:introExistPN}, and (G2) is to control $\ulh^1_\ldh F$, cf. Equation~\ref{eq:bup} on page~\pageref{eq:bup}. Unfortunately, (G1) remains a little mysterious. Its easy to say how it is used in the proof (it produces inclusions $F(X) \subseteq F_\ldh(X) \subseteq \prod_{x \in X} F(x)$, cf. Lem.\ref{lemm:ldhcdd}, Prop.\ref{prop:ldhTraces}), but not why it should be a necessary ingredient. Cf. also Remark~\ref{rema:whyGersten} for more on this.
 
The class of presheaves covered by the hypotheses of Theorem~\ref{theo:mainTheorem} is reasonably large. Any cohomology theory representable in the motivic stable homotopy category is invariant under universal homeomorphisms (at least with $\ZZ[\tfrac{1}{p}]$-coefficients), \cite{EK18}, \cite[Lem.3.15]{CD15}. It is quite common for cohomology theories in algebraic geometry to have some kind of trace or transfer morphisms, see Example~\ref{exam:traces} for a list of examples. Many cohomology theories in algebraic geometry satisfy (G1), cf. Gersten's Conjecture for algebraic $K$-theory \cite[Thm. 7.5.11]{Qui73}, or Cousin complexes in the Bloch-Ogus-Gabber theorem \cite{CTHK}. In particular, the cohomology theories representable in the motivic stable homotopy category that we are interested in satisfy (G1), \cite[Rmk.3.2]{KM18}, \cite[Thm.3.3.1]{Kel17}.

Condition (G2) seems newer. It is true for algebraic $K$-theory, or more generally, for nilpotent invariant theories commuting with filtered colimits which satisfy ``Milnor'' excision, \cite[Lem.3.5]{KM18}. The author is currently working with Elmanto, Hoyois, and Iwasa to prove ``Milnor'' excision for $SH$ with $\ZZ[\tfrac{1}{p}]$-coefficients. Cf. also Bhatt and Mathew's newly minted arc-topology \cite[Thm.1.6]{BM18}.

%\subsection*{The problem---non-Nagata-ness}

\textbf{The problem---non-Nagata-ness.} One of the first things one might try when comparing the cohomology of a finer topology $\lambda$ with a coarser one $\sigma$ is the change of topology spectral sequence
\[ H^p_\sigma(X, (\ulh_\lambda^q F)_\sigma) \implies H^{p+q}_\lambda(X, F_\lambda). \]
If one can show that $H^p_\lambda(P, F_\lambda) = 0$ (for $p > 0$ and $F_\lambda(P) = F(P)$) for schemes $P$ in a family inducing a conservative family of fibre functors of the $\sigma$-topos, it follows that $(\ulh_\lambda^qF)_\sigma = 0$ (for $p > 0$ and $F_\sigma = F_\lambda$), the spectral sequence collapses, and one is done. If $(\sigma, \lambda) = ($Zariski, étale$)$ then this would be to show that $H^p_{\et}(-, F_{\et})$ vanishes on all local rings. In our setting where $(\sigma, \lambda) = (\cdh, \ldh)$ it amounts to showing that $H_\ldh^p(-, F_\ldh)$ vanishes on finite rank hensel valuation rings, cf. Appendix B.%Section~\ref{sec:nonnoeth}. 

To prove this vanishing, one would like to use a structure of trace morphisms on $F$ (as formalised in Definition~\ref{def:traces}) and the well-known fact that every $\ldh$-covering of (the spectrum of) a hensel valuation ring is refinable by a finite flat surjective morphism of degree prime to $l$, via something like the following well-known lemma:

\begin{lemm}{(Lemma~\ref{lemm:hcVanish}, {\cite[Lemma 2.1.8]{Kel17}})} \label{lemm:hcVanishIntro}
Suppose that $F$ is a $\zll$-linear presheaf with traces in the sense of Definition~\ref{def:traces}, and $f: Y \to X$ a finite flat surjective morphism of degree prime to $l$. Then the complex
\[ 0 \to F(X) \to F(Y) \to F(Y \times_{X} Y) \to F(Y \times_{X} Y \times_{X} Y) \to \dots \]
is exact. 
\end{lemm}

The problem is that finite flat algebras over hensel valuation rings are not necessarily products of hensel valuation rings. We can try to return to hensel valuation rings by normalising, cf. Lemma~\ref{lemm:normHvr}, but then this takes us out of the category of finite algebras. To summarise: 

\begin{prob}
Given a hensel valuation ring $R$ and a finite faithfully flat $R$-algebra $R \to A$, in general, there is no $A$-algebra $A'$ such that $R \to A'$ is finite faithfully flat and $A'$ is a product of hensel valuation rings.
\end{prob}

The normalisation $R \to \widetilde{A}$ is a product of hensel valuation rings, and faithfully flat, but unless $\Frac(\widetilde{A}) / \Frac(R)$ is finite separable and $R$ is discrete \cite[Chap.6,Sec.8,No.5,Thm.2,Cor.1]{Bou64}, the morphism $R \to \widetilde{A}$ is in general no longer finite, not even for a general discrete valuation ring $R$, cf. Example~\ref{exam:nonNagata}. \\

%\subsection*{The solution---pseudo-normalisations}

\textbf{The solution---pseudo-normalisations.} The observation which rescues us is that $\zpi$-motivic cohomology, and more generally the sheaves we are interested in, are invariant under universal homeomorphism. Since we can restrict our attention to finite rank hensel valuation rings, cf. Corollary~\ref{coro:finiteRankConservative}, and all residue field extensions of a finite extension of valuation rings are finite, we don't have to normalise to catch all the information in the normalisation that we need. 

\begin{lemm}{(See Lemma~\ref{lemm:subNormalisation})} \label{lem:introExistPN}
Suppose $R$ is a finite rank hensel valuation ring and $R \to A$ a finite faithfully flat algebra. Then there exists an $A$-algebra $A'$ such that $R \to A'$ is finite faithfully flat and $\Spec(\widetilde{A'}) \to \Spec(A')$ induces an isomorphism on underlying topological spaces and all residue fields.
\end{lemm}

Since our sheaves don't distinguish between $A'$ and the (product of) valuation ring(s) $\widetilde{A'}$, we can use $A'$ as though it were a valuation ring. To work with this lemma we introduce the notion of pseudo-normalisation.

\begin{defi}{({See Def.\ref{defi:pn}})}
Let $A$ be a ring and $A \to B$ an $A$-algebra. Define a \emph{pseudo-integral closure}%
\footnote{If one prefers names that explain the meaning, we find \emph{finite cduh-local integral closure} and \emph{finite cduh-local normalisation} to be the most accurate.} %
of $A$ in $B$ to be a \emph{finite} sub-$A$-algebra 
\[ A \to B^{\pic} \subseteq B^{\ic} \subseteq B \]
of the integral closure $B^{\ic}$ of $A$ in $B$ such that $\Spec(B^\pic) \to \Spec(B^\ic)$ induces an isomorphism on topological spaces and residue fields. A \emph{pseudo-normalisation} of $A$ is a pseudo-integral closure of $A$ in its normalisation 
\[ A \to \breve{A} \subseteq (A^\red)^\sim. \]
\end{defi}

So the lemma above now becomes:

\begin{lemm}{(See Lemma~\ref{lemm:subNormalisation})}
Every finite faithfully flat algebra $A$ over a finite rank hensel valuation ring admits a pseudo-normalisation $A \to \breve{A}$.
\end{lemm}

%\subsection*{Outline}

\textbf{Outline.} In Section~\ref{sec:uh} (resp. \ref{sec:traces}, resp. \ref{sec:hvr}) we recall some well-known material on \emph{universal homeomorphisms} (resp. \emph{presheaves with traces}, resp. \emph{hensel valuation rings}). An interesting observation is that a morphism of schemes becomes an isomorphism of $\ldh$-sheaves under Yoneda if and only if it is a universal homeomorphism, Cor.\ref{coro:ldhUhIso}, Rem.\ref{rema:ldhUhIso}, (valid for any $l \neq p$). The $h$-version of this statement is well-known and due to Voevodsky \cite{Voe96} for excellent noetherian schemes and Rydh \cite{Ryd10} in general.

In Section~\ref{sec:pn} we introduce the notion of \emph{pseudo-integral closure} and \emph{pseudo-normalisation}, and develop some basic properties.

%Sections~\ref{sec:pn}, \ref{sec:hvrldh}, \ref{sec:tracesFldh}, \ref{sec:cdhldhComp}

Sections~\ref{sec:hvrldh}-\ref{sec:cdhldhComp} (pages~\pageref{sec:hvrldh}-\pageref{sec:cdhldh}) contain our proof.

In Section~\ref{sec:hvrldh} the condition (G1) appears, and we use it to show that for hensel valuation rings $R$ we have $F(R) \cong F_\ldh(R)$, Prop.~\ref{prop:FisFldhonHvr}.

In Section~\ref{sec:tracesFldh} we continue using (G1) to show that traces on $F$ induce traces on $F_\ldh$, Prop.~\ref{prop:ldhTraces}.

Section~\ref{sec:cdhldhComp} (pages~\pageref{sec:cdhldhComp}-\pageref{sec:cdhldh}) contains our main theorem (Theorem~\ref{theo:mainTheorem}) %---that the $\cdh$- and $\ldh$-cohomologies of nice presheaves agree)
 and the rest of its proof.

%Section~\ref{sec:tracesOnH1} (pages~\pageref{sec:tracesOnH1}-\pageref{sec:cdhldh}) contains a technical result that traces on a presheaf $F$ pass to some associated presheaves $E$ and $D$ (Definition~\ref{defi:DE}). This is needed to control $\ulh^1_\ldh F$ if (G1) is satisfied instead of (G2) in Theorem~\ref{theo:mainTheorem}.

In Appendix A %Section~\ref{sec:cdhldh} 
we recall the definitions of the $\cdh$- and $\ldh$-topologies, and observe that for finite dimensional noetherian schemes, the class of \emph{finite dimensional} hensel valuation rings induces a conservative family of fibre functors.

In Appendix B %Section~\ref{sec:nonnoeth} 
we confirm that everything we need passes from finite type separated $S$-schemes to all quasi-compact separated $S$-schemes just as one would expect. \\

%\subsection*{Acknowledgements}
\textbf{Acknowledgements.} The idea for this strategy is from 2013, but it took five years to realise how to get around the non-Nagata-ness of valuation rings. This proof owes its existence to Veronika Ertl. It was during discussions with her about de Rham-Witt differentials of valuation rings that I had the idea to use pseudo-normalisations. I thank Giuseppe Ancona, Matthew Morrow, and Simon Pepin-Lehalleur for comments that improved the writing style. I also thank an anonymous referee for helping catch a mathematical error in a previous version.

%\begin{conv}[Presheaves] \label{conv:preshv}
\textbf{Convention.} \label{convention} We work with a base scheme $S$ which will almost always be separated and noetherian, and often of finite dimension. We write:
\renewcommand{\descriptionlabel}[1]{%
  \hspace\labelsep \upshape\bfseries #1.%
}
\begin{description}
 \item[] $\Sch_S$ for the category of separated finite type $S$-schemes. Unless otherwise indicated, \emph{presheaf} means a presheaf on $\Sch_S$ which is extended to the category 
 \item[] $\SCH_S$ of all quasi-compact separated $S$-schemes by left Kan extension. That is $F(T) = \varinjlim_{T \to X \to S} F(X)$ where the colimit is over factorisations through $X \in \Sch_S$. Another way of saying this, (when $S$ is quasi-compact and quasi-separated), is that 
 
 \item[] \emph{presheaf} means an additive functor $F: \SCH_S^\op \to \Ab$ that commutes with filtered colimits. For more on this see Appendix B. \\%Section~\ref{sec:nonnoeth}. \\

We use
 \item[] $h_-: C \to \PreShv(C); X \mapsto h_X$ for the Yoneda functor, and 
 \item[] $(-)_\tau: \PreShv(C) \to \Shv_\tau(C); F \mapsto F_\tau$ for the sheafification functor. We write 
 \item[] $\underline{H}_\tau^n F$ for the presheaf $H^n_\tau(-, F_\tau)$ associated to a presheaf $F$. This paper's topologies (principally $\cdh, \ldh, \fpsl$) are classically defined on $\Sch_S$ for $S$ noetherian. We extend them to $\SCH_S$ in the canonical way, cf. Appendix B.
 \item[] $Q(A)$ is the total ring of fractions of a ring $A$. We will only ever apply this to reduced rings.
 \item[] $l$ and $p$ are always distinct primes.
\end{description}
%\end{conv}

\section{Universal homeomorphisms} \label{sec:uh}

In this section we recall some basic material on universal homeomorphisms ($\uh$), and show that $\ldh$-sheafification preserves $\uh$-invariances. For some background on the $\ldh$-topology for noetherian, and non-noetherian schemes, we direct the reader to Appendix A and Appendix B, respectively.

\begin{defi} \label{defi:uh}
A morphism of schemes $f: Y \to X$ is a \emph{universal homeomorphism} or $\uh$ if it satisfies the following equivalent conditions.
\begin{enumerate}
 \item For every $X$-scheme $T \to X$, the morphism $T \times_X Y \to T$ induces a homeomorphism on the underlying topological spaces.

 \item {\cite[Tag 01S4, Tag 01S3]{Stacks}} $f$ is a homeomorphism, and for every $y \in Y$, the field extension $k(y) / k(f(y))$ is purely inseparable.

 \item \label{defi:uh:isi} {\cite[Cor.18.12.11]{EGAIV4}}%, \cite[Tag 04DF]{Stacks}} 
 $f$ is integral, %
% \footnote{Recall that a morphism $f: Y \to X$ is integral if it is affine, and for every open affine $U \subset X$ every $a \in \Gamma(f^{-1}U, \OO_{f^{-1}U})$ satisfies a monic polynomial with coefficients in $\Gamma(U, \OO_{U})$, \cite[Tag 01WG]{Stacks}. There are no finiteness conditions.} %
surjective and universally injective.

 \item {\cite[Tag 01WM]{Stacks}} $f$ is affine, surjective, and universally injective.
\end{enumerate}
\end{defi}

Note there are no finiteness conditions. There are universal homeomorphisms that are not of finite type, pertinantly Example~\ref{exam:nonNagata}, and finite universal homeomorphisms that are not finitely presented. %\footnote{
For example $A = \varinjlim_{n \to \infty} \FF_p[x_1, \dots, x_n] / \langle x_ix_j : i \neq j \rangle$ and the map $\phi: A {\to} A; x_1 {\mapsto} x_1^p$. To generate the kernel of $A[y] \to A; \sum a_m y^m \mapsto \sum \phi(a_m) x_1^m$, in addition to $y^p - x_1$ one needs all the $yx_n$ for $n > 1$.%}

\begin{lemm} \label{lemm:subuh}
Suppose that $Y \to X$ is a universal homeomorphism, and $\mathcal{A} \subseteq \OO_Y$ is a sub-$\OO_X$-algebra. Then both of $Y \to \uSpec(\mathcal{A}) \to X$ are universal homeomorphisms. 
\end{lemm}

\begin{proof}
By Definition~\ref{defi:uh}\eqref{defi:uh:isi} it suffices to prove that $Y \to \uSpec(\mathcal{A})$ is surjective. But this follows from integrality and injectivity of $\mathcal{A} \subseteq \OO_Y$, \cite[{Tag 00GQ}]{Stacks}.
\end{proof}

Recall that a morphism of schemes $Y \to X$ is \emph{completely decomposed} if it induces a surjection $Y(K) \to X(K)$ of $K$-points for every field $K$.

\begin{lemm} \label{lemm:uhldh}
Let $S$ be a noetherian separated scheme of positive characteristic $p$. Every universal homeomorphism $Y \to X$ in $\Sch_S$ is refineable by a composition $Y_1 \to Y_0 \to X$ such that $Y_0 \to X$ is proper and completely decomposed, and $Y_1 \to Y_0$ is finite flat surjective and locally of degree a power of $p$.
\end{lemm}

\begin{proof}
We reproduce the standard proof using Raynaud-Gruson flatification and induction on the dimension. The initial case is dimension $-1$, i.e., the empty scheme. Suppose that $Y \to X$ is a universal homeomorphism with $\dim X = n \geq 0$, and suppose that the statement is true for a target scheme of dimension $< n$. Replacing $X$ and $Y$ with their reductions, we can assume that $Y \to X$ is generically flat. In this case, by Raynaud-Gruson flatification \cite[Thm.5.2.2]{RG71}, there exists a nowhere dense closed subscheme $Z \subset X$ such that the strict transform $Y' \to \Bl_Z X$ is globally flat. On the other hand, as $Y \to X$ is an integral morphism of finite type, it is finite. Therefore $Y' \to \Bl_Z X$ is flat \emph{and} finite. Moreover, as $Y \to X$ is a universal homeomorphism, for each generic point $\xi \in X$ with corresponding point $\eta \in Y$, the finite field extension $k(\eta) / k(\xi)$ is purely inseparable. In particular, as $Y'$ and $\Bl_Z X$ are reduced, $Y' \to \Bl_Z X$ is locally of degree a power of $p$. Using the noetherian inductive hypothesis to choose a refinement $W_1 \to W_0 \to Z$ of the universal homeomorphism $Y\times_X Z \to Z$, we have produced the desired refinement $W_1 \sqcup Y' \to W_0 \sqcup \Bl_Z X \to X$ of $Y \to X$.
\end{proof}

\begin{coro} \label{coro:ldhUhIso}
Let $S$ be a noetherian %separated %% DONT NEED S SEPARATED, ONLY $\Sch_S$ TO CONSIST OF SEPARATED S-SCHEMES
scheme of positive characteristic $p \neq l$ and $f: Y \to X$ a universal homeomorphism in $\Sch_S$. Then the image $(h_f)_\ldh \in \Shv_{\ldh}(\Sch_S)$ of $f$ is an isomorphism.%, where $h:\Sch_S {\to} \PreShv(\Sch_S)$ is the Yoneda embedding and $a_{\ldh}: \PreShv(\Sch_S) \to \Shv_\ldh(\Sch_S)$ is the sheafification functor.
\end{coro}

\begin{rema} \label{rema:ldhUhIso}
In fact, the converse is also true: \cite[Thm.3.2.9]{Voe96} and \cite[Thm.8.16]{Ryd10} say that the image of $(h)_\h: \Sch_S \to \Shv_\h(\Sch_S)$ is equivalent to the localisation of $\Sch_S$ at the class $\uh$. As $\uh$ is a right multiplicative system and satisfies the 2-out-of-6 property, it follows that a morphism in $\Sch_S$ is in $\uh$ if and only if it becomes an isomorphism in $\Sch_\h(\Sch_S)$, \cite[7.1.20]{KS06}.  Since $\Sch_S \to \Shv_\h(\Sch_S)$ factors through $\Shv_\ldh(\Sch_S)$, the converse of Corollary~\ref{coro:ldhUhIso} follows.
\end{rema}

\begin{proof}
Surjectivity is a result of $f$ being (refineable by) an $\ldh$-cover, Lemma~\ref{lemm:uhldh}. For injectivity, consider some $s, s' \in (h_Y)_\ldh(T)$ sent to the same element of $(h_X)_\ldh(T)$. Let $T' \to T$ be an $\ldh$-cover such that $s, s'$ can be represented by some morphisms $t, t': T' \to Y$ in $\Sch_S$ with $T'$ reduced. Possibly refining $T'$, we can assume that $ft = ft'$. Let $\eta_1, \dots, \eta_n$ be the generic points of $T'$. As $f: Y \to X$ is a universal homeomorphism, it is injective on the underlying topological space, and all residue field extensions are purely inseparable. Consequently, $\hom(\amalg \eta_i, Y) \to \hom(\amalg \eta_i, X)$ is injective, so $t|_{\amalg \eta_i} = t'|_{\amalg \eta_i}$. But $Y \to X$ is separated, and $T'$ reduced, so % the image of $T' \to Y \times_X Y$ lands inside the diagonal, i.e., 
$t = t'$.
\end{proof}

\begin{coro} \label{coro:Huhinv}
Let $S$ be a noetherian scheme of positive characteristic $p \neq l$. For any presheaf $F$ and $n \geq 0$, the associated presheaf $(\ulh^n_\ldh F)(-) = H_\ldh^n(-, F_\ldh)$ is $\uh$-invariant.
\end{coro}

\begin{proof}
Let $F_\ldh \to I^\bullet$ be an injective resolution in $\Shv_\ldh(\Sch_S)$% of the $\ldh$-sheaf $F_\ldh$ associated to $F$. 
. %
Then we have $(\ulh^n_\ldh F)(-) = H^n(\hom_{\Shv_\ldh}((h_-)_\ldh, I^\bullet))$. But Corollary~\ref{coro:ldhUhIso} says that $(h_-)_\ldh$ sends universal homeomorphisms to isomorphisms. Hence, the same is true of $\ulh^n_\ldh F$.
\end{proof}

\section{Traces and $\fpsl$-descent} \label{sec:traces}

This section contains material on presheaves with traces from \cite{Kel17} included for the convenience of the reader.

\begin{defi}
We abbreviate finite flat surjective to $\fps$, and finite flat surjective of degree prime to $l$ to $\fpsl$.
\end{defi}

\begin{defi}{({\cite[Def.2.1.3]{Kel17}})} \label{def:traces}
Let $\cS$ be a category of schemes admitting fibre products. A \emph{structure of traces} on a presheaf $F: \cS^{op} \to Ab$ is a morphism $\Tr_f: F(Y) \to F(X)$ for every $\fps$ morphism $f: Y \to X$, satisfying the following axioms.
\begin{enumerate}
 \item {(Add)} 
 We have $\Tr_{f \amalg f'} = \Tr_f \oplus \Tr_{f'}$ for every pair $Y \stackrel{f}{\to} X$, $Y' \stackrel{f'}{\to} X'$ of $\fps$ morphisms.
 \item{(Fon)}
 We have $\Tr_f \circ \Tr_g = \Tr_{f \circ g}$ for every composable pair $W {\stackrel{g}{\to}} Y {\stackrel{f}{\to}} X$ of $\fps$ morphisms.
 \item {(CdB)} 
 We have $F(p) \circ \Tr_f = \Tr_g \circ F(q)$ for every $\fps$ morphism $f: Y \to X$ and every cartesian square
 \[ \xymatrix{
W \times_X Y \ar[r]^-q \ar[d]_g & Y \ar[d]^f \\
W \ar[r]_p & X.
 } \]
\item {(Deg)} 
We have $\Tr_f \circ F(f) = d \cdot \id_{F(X)}$ for every $\fps$ morphism $f: Y \to X$ of constant degree $d$.
\end{enumerate} 
A presheaf equipped with a structure of traces is called a \emph{presheaf with traces}. A presheaf with traces taking values in the the category of $R$-modules for some ring $R$, is called a \emph{presheaf of $R$-modules with traces} or an \emph{$R$-linear presheaf with traces}.
\end{defi}

\begin{exam} \label{exam:traces} {({For these and more examples see \cite[Exam.2.1.4]{Kel17}})}\ 
\begin{enumerate}
 \item \label{exam:traces:const} Every constant additive presheaf has a unique structure of traces determined by the axioms (Add) and (Deg).
 
 \item The trace and determinant equip $(\OO, +)$ and $(\OO^*, *)$ respectively with a structure of traces, \cite[Exam.2.1.4(vii)]{Kel17}.

 \item If $F$ is a presheaf with traces, then the associated Nisnevich sheaf $F_{\Nis}$ inherits a unique structure of traces compatible with the canonical morphism $F \to F_{\Nis}$, \cite[Cor.2.1.13]{Kel17}. This is essentially because a finite algebra over a hensel local ring is a product of hensel local rings.

 \item Pushforward of vector bundles would induce a structure of traces on higher $K$-theory $K_n$ and homotopy invariant $K$-theory $KH_n$ for every $n \in \ZZ$, except (Deg) is only satisfied Zariski locally. The Nisnevich sheafifications of these sheaves $(K_n)_{\Nis}$ and $(KH_n)_{\Nis}$ have canonical structures of traces, cf. \cite[Proof of Lem.3.1]{Kel14}.
%
% \item This article would be even shorter if we could prove that the $\uh$-sheaf, cf. Definition~\ref{defi:uh}, associated to a presheaf with traces has a structure of traces. The obstracle to proving this seems connected to the obstacle to proving that the $\cdh$-sheafification of a presheaf with traces inherits a structure of traces.
\end{enumerate}
\end{exam}

\begin{lemm}{({cf.\cite[Lemma 2.1.8]{Kel17}})} \label{lemm:hcVanish}
Suppose that $F$ is a $\zll$-linear presheaf with traces, and $X_\bullet \to X_{-1}$ is a simplicial $X_{-1}$-scheme such that each $X_{n+1} \to (\cosk_n X_\bullet)_{n+1}$ is an $\fpsl$-morphism of constant degree, for example, $X_{n-1} = Y^{\times_X n}$ for some $\fpsl$-morphism of constant degree $Y \to X$. Then the complex
\[ 0 \to F(X_{-1}) \to F(X_0) \to F(X_1) \to F(X_2) \to \dots \]
is exact. Here the morphisms are $\sum (-1)^i d_i$.

Consequently, $F$ is a $\fpsl$-sheaf, and we have both $\check{H}^n_{\fpsl}(-, F) = 0$, and $H^n_{\fpsl}(-, F_{\fpsl}) = 0$ for $n > 0$.
\end{lemm}

\begin{proof}
For each $0 \leq i < j \leq n$ we have the commutative diagram
\[ \xymatrix{
X_{n+1} \ar[r]^-a \ar[drr]_{d_i} \ar@/^12pt/@<1.5ex>[rrr]^{d_j} & (\cosk_n X_\bullet)_{n+1} \ar[r]^-b & X_n \times_{X_{n-1}} X_n \ar[r]^-{pr_2} \ar[d]_-{pr_1} & X_n \ar[d]^{d_i} \\
&& X_n \ar[r]_{d_{j-1}} & X_{n-1} .
} \]
All morphisms are $\fpsl$-morphisms of constant degree, \cite[\href{https://stacks.math.columbia.edu/tag/01GN}{Tag 01GN}]{Stacks}. Setting, $D_m = \tfrac{\deg(X_m \to X_{-1})}{\deg(X_{m-1} \to X_{-1})} = \deg(X_m \stackrel{d_i}{\to} X_{m-1})$, the composition $ba$ is of degree $D_{n+1}/ D_n$. Now, it follows from the above diagram, that
\begin{align*}
 \tfrac{1}{D_{n+1}} Tr_{d_i} F(d_j) 
 \stackrel{(Fon)}{=} \tfrac{1}{D_{n+1}} Tr_{pr_1} Tr_{ba} F(ba) F(pr_2) 
 &\stackrel{(Deg)}{=} \tfrac{1}{D_{n}} Tr_{pr_1} F(pr_2) 
 \\&\stackrel{(CdB)}{=} \tfrac{1}{D_{n}} F(d_{j-1}) Tr_{d_i}
 \end{align*}
In particular, $\tfrac{1}{D_\bullet} Tr_{d_0}$ is a chain homotopy between the zero morphism and the identity morphism of the chain complex in the statement.

For the second statement, notice that any $\fpsl$-morphism is refinable by an $\fpsl$-morphism of constant degree. Hence, the \v{C}ech cohomologies of these two classes of morphisms are the same, and moreover, they generate the same topology. Since the colimit over all hypercovers calculates sheaf cohomology, \cite[Thm.7.4.1(2)]{SGA42}, vanishing of sheaf cohomology also follows from exactness of the sequence.
\end{proof}

\section{Hensel valuation rings} \label{sec:hvr}

Recall that an integral domain $R$ is a \emph{valuation ring} if for all nonzero $a \in \Frac(R)$, we have $a \in R$ or $a^{-1} \in R$. Equivalently, the set of ideals of $R$ is totally ordered by inclusion,  \cite[Chapitre VI, \S1.2, Théorème 1]{Bou64}.

A valuation ring $R$ is a \emph{hensel valuation ring} or \emph{hvr} if it extends uniquely to every finite field extension, \cite[\S4.1]{EP05}. That is, for every finite field extension $L / \Frac(R)$, there exists a unique valuation ring $R' \subseteq L$ such that $L = \Frac(R')$ and $R = L \cap R'$. A valuation ring is a hvr if and only if it satisfies Hensel's Lemma, \cite[Thm.4.1.3]{EP05}.

We will frequently use the following lemma. Recall that $Q(-)$ denotes the total ring of fractions.

\begin{lemm} \label{lemm:normHvr}
Let $R$ be a hvr, let $R \to A$ be a finite $R$-algebra, and let $Q(A^\red) \to L$ be a finite morphism with $L$ reduced. Then the integral closure $A^\ic$ of $A^\red$ in $L$ is a product of valuation rings. If $A$ and $L$ are integral domains the induced morphism $\Spec(A^\ic) \to \Spec(A)$ is a homeomorphism. 
\end{lemm}

%$A^\ic = \{ l \in L : f(l) = 0$ for some monic $f(T) \in A[T]\}$.

\begin{proof}
The integral closure $A^\ic$ is the product of the integral closures of the images of $A^\red$ in the residue fields of $L$, so we can assume $A$ is an integral domain and $L$ a field. Replacing $R$ with its image%
\footnote{The image is an integral domain because $A$ is, and its ideals are the ideals of $R$ containing the kernel of $R \to A$, so they are totally ordered. That is, the image of $R$ in $A$ is certainly a valuation ring.} %
in $A$, we can assume $R \to A$ is injective. Note that since $A$ is finite, the integral closure of $A$ in $L$ is equal to the integral closure of $R$ in $L$. %Apply \cite[\href{https://stacks.math.columbia.edu/tag/0308}{Tag 0308}]{stacks-project} to $R \to Q(A) \to L$ to get see that $IC(L/A) \subseteq IC(L/IC(Q(A)(R)) = IC(L/R)$. The opposite inclusion is obvious. 
Now the first claim follows from the facts that the integral closure of a valuation ring is the intersection of the extensions \cite[Cor.3.1.4]{EP05}, and since $R$ is henselian, by definition there is a unique extension to $L$. Now, $R \subseteq A \subseteq A^\ic$ are integral extensions of rings so $\Spec(A^\ic) \to \Spec(A) \to \Spec(R)$ are surjective \cite[\href{http://stacks.math.columbia.edu/tag/00GQ}{Tag 00GQ}]{Stacks}. Since $R \subseteq A \subseteq A^\ic$ are integral extensions, the incomparability property implies that $\Spec(A^\ic) \to \Spec(A) \to \Spec(R)$ are injective. Finally, since the prime ideals of $A^\ic$ and $R$ are totally ordered, the bijection $\Spec(A^\ic) \to \Spec(A)$ is a homeomorphism. 
\end{proof}

\begin{lemm} \label{lemm:uhClosed}
Suppose that $R$ is a hvr of characteristic $p$ with fraction field $K$, and suppose that $K \to K'$ is a purely inseparable extension, and $R'$ the integral closure of $R$ in $K'$. Then $\Spec(R') \to \Spec(R)$ is a universal homeomorphism.
\end{lemm}

%This lemma will be used in Proposition~\ref{prop:getfpslIsfsl} to show that the $\fpsl$ \v{C}ech cohomology of a finite rank hvr can be calculated using $\fpsl$-covers that are generically étale.

%\begin{rema}
%The proof seems to work more generally for any $R$ normal, $R'$ integral, and $R \subseteq R'$ finite.
%\end{rema}

\begin{proof}
The extension $R \to R'$ is integral by definition, and $\Spec(R') \to \Spec(R)$ is surjective as it is dominant and satisfies the going up property. %\cite[Tag 01WM]{Stacks}
So it remains to see that it is injective, and each extension of residue fields is purely inseparable. As $R$ is henselian, $R'$ is a valuation ring, and so its poset of primes is totally ordered. By the incomparability property, it follows that there is exactly one prime of $R'$ lying over any prime of $R$. Let $\p \subset R$ be a prime and $\p' \subset R'$ the prime lying over it. Localising at $\p$ we can assume both are maximal ideals. Then a given $a \in k(\p') = R' / \p'$, lifts to some $b \in R'$, and since $K' / K$ is purely inseparable, $b^{p^i} \in K$ for some $i$.  Now $R'$ is the integral closure of $R$ in $K'$, so $b$ satisfies some monic $f(T) \in R[T]$. But then $b^{p^i}$ satisfies the monic $f(T)^{p^i}$. Since valuation rings are normal, it follows that $b^{p^i} \in R$. Consequently, $a^{p^i} \in k(\p) = R / \p$. So $k(\p') / k(\p)$ is purely inseparable.
\end{proof}

\section{Pseudo-normalisation} \label{sec:pn}

One of the obstacles to using valuation rings to study finite type morphisms of noetherian schemes is non-Nagata-ness: Suppose $R$ is a hvr and $R \subset A$ a finite extension with $A$ an integral domain. Then the normalisation $\widetilde{A}$ of $A$ is a valuation ring, Lem~\ref{lemm:normHvr}. If $\Frac(A) / \Frac(R)$ is finite separable and $R$ is discrete then the morphism $R \to \widetilde{A}$ is also finite \cite[Chap.6,Sec.8,No.5,Thm.2,Cor.1]{Bou64}. However, in general, the morphism $R {\to} \widetilde{A}$ may not be finite. 

\begin{exam} {\cite[Chap.6,Sec.8,Exercise 3b]{Bou64}} \label{exam:nonNagata}
Let $k = \FF_p(X_n)_{n \in \NN}$, and equip $K = k(U, V)$, with the (discrete) valuation induced by $\phi: k(U, V) \to k((U)); V \mapsto \sum_{i = 0}^\infty X^p_i U^{ip}$. Let $K' = K(V^{1/p})$. Then, (exercise), the unique extension of valued fields $K'/K$ induces an isomorphism on both value groups and residue fields. Consequently, the corresponding morphism $R \to R'$ of discrete valuation rings is not finite, and so the normalisation $R'$of $R[V^{1/p}]$ is not finite over $R[V^{1/p}]$.
\end{exam}

On the other hand, at least if the rank of $R$ is finite, if we take a large enough member $\breve{A}$ of the filtered poset of finitely generated sub-$A$-algebras of $\widetilde{A}$, the induced morphism $\Spec(\widetilde{A}) \to \Spec(\breve{A})$ will be a universal homeomorphism (and even completely decomposed). Hence, as far as $\uh$-sheaves (and even $\cdh$-sheaves, cf.~\cite[Lem.2.9]{HK18}) are concerned, the normalisation might as well be finite.

To formalise this phenomenon, we introduce the notions of pseudo-integral closure, pseudo-normalisation, and pseudo-hvr.

%Total ring of fractions: https://stacks.math.columbia.edu/tag/02LV
%If A is reduced, and has finitely many minimal primes, then $Q(A) = \prod A_\p$.

\begin{conv}{({\cite[Cor.6.3.8]{EGAII}, \cite[Tag.035N,035P]{Stacks}})}
By \emph{normalisation} $\widetilde{A}$ of a ring $A$ with finitely many minimal primes, we mean the integral closure of its associated reduced ring $A^{\red}$ in the total ring of fractions of $Q(A^{\red})$ of $A^\red$.
\end{conv}

Note, that usually one only talks of normalisations of reduced rings.

\begin{defi} \label{defi:pn}
Let $A$ be a ring and $A \to B$ an $A$-algebra. Define a \emph{pseudo-integral closure} of $A$ in $B$ to be a \emph{finite} sub-$A$-algebra \mbox{$A {\to} B^{\pic} {\subseteq} B^{\ic}$} of the integral closure $B^{ic} \subseteq B$ of $A$ in $B$ such that $\Spec(B^\pic) \to \Spec(B^\ic)$ is a %
completely decomposed %
universal homeomorphism. 

A \emph{pseudo-normalisation} of a ring $A$ with finitely many minimal primes is a pseudo-integral closure of $A$ in its normalisation $A \to \breve{A} \subseteq (A^\red)^\sim$. We will write $\PIC(B/A)$ and $\PN(A)$ for the set of pseudo-integral closures, and pseudo-normalisations respectively.

A \emph{pseudo-hvr} is an integral domain $A$ such that $\widetilde{A}$ is an hvr, and $A$ is a pseudo-normalisation of itself, $A = \breve{A}$. That is, $\Spec(\widetilde{A}) \to \Spec(A)$ is a completely decomposed universal homeomorphism.
\end{defi}

Even if pseudo-integral closures and pseudo-normalisations exist they are certainly not unique in general. If the normalisation resp. integral closure is finite, then it is the final pseudo-normalisation resp. pseudo-integral closure.

The following lemma contains the basic facts we need about pseudo-normalisations $\breve{A}$, at least for finite rings $A$ over a finite rank hvr $R$: \ref{lemm:subNormalisation:2}. they exist, \ref{lemm:subNormalisation:1-5}., \ref{lemm:subNormalisation:Fun}. they are (ind-)functorial for dominant morphisms, and \ref{lemm:subNormalisation:flat}. they preserve flat morphisms.

\begin{lemm} \label{lemm:subNormalisation}
Let $R$ be a finite rank hvr. 
\begin{enumerate}
 \item \label{lemm:subNormalisation:1} Let $A \to B$ be a ring homorphism. If $B^\pic \subseteq B^\ic$ is a pseudo-integral closure, then any finitely generated subalgebra $B^\pic \subseteq B' \subseteq B^\ic$ is also a pseudo-integral closure.
 
 \item \label{lemm:subNormalisation:1-5} If the collection of pseudo-integral closures of a ring homomorphism $A {\to} B$ is non-empty, then it is a filtered poset and $B^\ic = \bigcup_{B^\pic \in \PIC(B/A)} B^\pic$.

 \item \label{lemm:subNormalisation:2} If $A$ is a finite $R$-algebra, and $Q(A^\red) \to K$ finite with $K$ reduced, the poset $\PIC(K/A)$ of pseudo-integral closures is nonempty. In particular, the conclusion of part \eqref{lemm:subNormalisation:1-5} holds. 
  
 \item \label{lemm:subNormalisation:Fun} If $\phi: A \to B$ is a morphism of finite $R$-algebras such that $\Spec(\phi)$ sends generic points to generic points, and 
\[ \xymatrix@R=12pt{
K \ar[r] & L \\
Q(A^\red) \ar[r] \ar[u] & Q(B^\red) \ar[u]
} \]
a square of finite morphisms of reduced rings, then for any $A^\pic \in \PIC(K/A)$ there is a $B^\pic \in \PIC(L/B)$ compatible with the above square. In other words, the canonical morphism of integral closures $A^\ic \to B^\ic$ of $A \to B$ in $K \to L$ induces a morphism of ind-objects 
\[ \ind_{\PIC(K/A)} A^\pic \to \ind_{\PIC(L/B)} B^\pic. \]

 \item \label{lemm:subNormalisation:flat} If $\phi: A \to B$ is a finite flat morphism of finite $R$-algebras, then there exists an inclusion $\breve{A} \subseteq \breve{B}$ of pseudo-normalisations compatible with $\phi$, which is flat.
%let $\PN(\phi)$ denote the set of finite inclusions $\breve{A} \subset \breve{B}$ compatible with $\widetilde{A} \subseteq \widetilde{B}$ which are flat \emph{and} such that $\widetilde{A} \otimes_{\breve{A}} \breve{B} \to \widetilde{B}$ is injective. Then both the canonical forgetful functors 
% \[ \PN(\phi) \to \PN(A), \PN(B); \qquad (\breve{A} \subseteq \breve{B}) \mapsto \breve{A}, \breve{B}  \]
% are cofinal, and elements of $\PN(\phi)$ are preserved by tensor product with morphisms in $\PN(A)$.
\end{enumerate}
\end{lemm}

%We imagine there may be a pseudo-integral closure generalisation of part \eqref{lemm:subNormalisation:flat} but we don't need it. In fact, we only use part \eqref{lemm:subNormalisation:flat} in Section~\ref{sec:tracesOnH1} so if (G2) is satisfied in Thm.\ref{theo:mainTheorem} we don't need part \eqref{lemm:subNormalisation:flat} at all.

\begin{proof}\ 
\begin{enumerate}
 \item Since $B^\ic {\supseteq} B^\pic$ is integral so is $B^\ic {\supseteq} B'$, and consequently, $\Spec(B^\ic) \to \Spec(B')$ is surjective, \cite[\href{http://stacks.math.columbia.edu/tag/00GQ}{Tag 00GQ}]{Stacks}. But $\Spec(B^\ic) \to \Spec(B^\pic)$ is a homeomorphism, so $\Spec(B^\ic) \to \Spec(B')$ is also injective. Finally, for any point $x \in \Spec(B^\ic)$ with images $y,z$ in $\Spec(B'), \Spec(B^\pic)$, %pure inseparability of $k(x) / k(z)$ implies pure inseparability of $k(x) / k(y)$.
The isomorphism $k(x) \cong k(z)$ implies an isomorphism $k(x) \cong k(y)$.

 \item Follows from part \eqref{lemm:subNormalisation:1}.

 \item Let $A^\ic$ be the integral closure of $A^\red$ in $K$. Since $A^\red$ is reduced with finitely many minimal primes, $Q(A^\red)$ is a finite product of fields, \cite[Tag 02LV]{Stacks}, so $K$ is also a finite product of fields. Hence $A^\ic$ is the product of the integral closures in the residue fields of $K$, and therefore it suffices to consider the case $A$ and $K$ are both integral domains. Replacing $R$ with its image in $A$, we can also assume $R \to A$ is injective. Now $R \subseteq A^\ic$ is an extension of valuation rings, Lem.\ref{lemm:normHvr}. As $R \subseteq A^\ic$ is generically finite, for every prime $\p \subset R$, the extension $k(\p) \subset k(\q)$ is finite, where $\q \subset A^\ic$ is the prime lying over $\p$, \cite[Tag 0ASH]{Stacks}. For each prime of $R$, choose a set of generators of the corresponding finite field extension, lift them to $A$, and let $A^\pic$ be the sub-$R$-algebra that they generate. Certainly, $A^\pic$ is a finite $R$-algebra. The morphism $A^\pic \subseteq A^\ic$ is an integral ring extension so 
 %, and birational by construction, so by the going up theorem 
 $\Spec(A^\ic) \to \Spec(A^\pic)$ is surjective, \cite[\href{http://stacks.math.columbia.edu/tag/00GQ}{Tag 00GQ}]{Stacks}. As $R$ is henselian, $\Spec(A^\ic) \to \Spec(R)$ is a homeomorphism, and we conclude that $\Spec(A^\ic) \to \Spec(A^\pic)$ is also a homeomorphism. By construction it induces isomorphisms on all field extensions. Hence, it is a completely decomposed universal homeomorphism. So $\PIC(K/A)$ is nonempty.

 \item This follows from Part~\ref{lemm:subNormalisation:2}---Existence, and Part~\ref{lemm:subNormalisation:1}---Closure under subextension: Choose any $B^\pic_0 \in \PIC(B)$ and take $B^\pic$ to be the sub-$B$-algebra of $L$ generated by the image of $A^\pic$ and $B^\pic_0$. 

 \item %The tensor product claim follows from part 1: If $\breve{A_0} \to \breve{B_0}$ is in $\PN(\phi)$ and $\breve{A_0} \subseteq \breve{A_1}$ in $\PN(A)$ then $\breve{A_1} \otimes_{\breve{A_0}} \breve{B_0} \subseteq \widetilde{A} \otimes_{\breve{A_0}} \breve{B_0}$ by flatness in the definition of $\PN(\phi)$. Now as long as $\PN(\phi)$ is nonempty, $\PN(\phi) \to \PN(A)$ being cofinal follows from $\PN(A)$ being filtered and the pullback claim. Since nonemptiness will follow from cofinality of $\PN(\phi) \to \PN(B)$, it remains only to prove this latter. 
We will construct the following diagram. Note all morphisms except possibly the ones with source $A$ and $B$ are inclusions. %We write $\tilde{A}$ and $\tilde{B}$ instead of $(A^\red)^\sim$ and $(B^\red)^\sim$ for compactness of notation.
%\[ \xymatrix{
%B \ar[r] & \breve{B}_0 \ar[r] & \breve{B}_1 \ar@/^12pt/[rr] & \breve{B}_2 \ar[r] & \breve{B}_3 \ar[r] & B' \ar[r] \ar@{}@<3pt>[r]^(0.4){= \langle \breve{B}_1, (A^\red)^\sim\rangle} & (B^\red)^\sim \\
%A \ar[u] \ar[rr] && \breve{A}_0 \ar[r] \ar[u]^(0.7){\textrm{f.t., not}}^{\textrm{necessarily}}^(0.3){\textrm{flat}} & \breve{A}_2 \ar[r] \ar[u]^(0.3){\textrm{f.t.flat}} \ar@{}[ur]|\square & \breve{A}_3 \ar[r] \ar[u]^(0.3){\textrm{f.t.flat}} \ar@{}[ur]|\square & (A^\red)^\sim \ar[u]^(0.3){\textrm{f.t.flat}} \ar[ur] 
%} \]
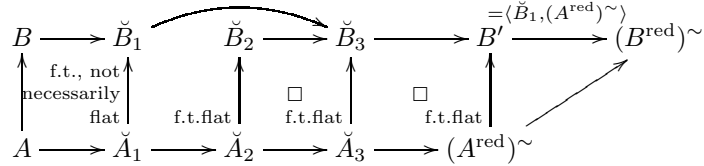
\[ \xymatrix{
B \ar[r]  & \breve{B}_1 \ar@/^12pt/[rr] & \breve{B}_2 \ar[r] & \breve{B}_3 \ar[r] & B' \ar[r] \ar@{}@<3pt>[r]^(0.4){= \langle \breve{B}_1, (A^\red)^\sim\rangle} & (B^\red)^\sim \\
A \ar[u] \ar[r] & \breve{A}_1 \ar[r] \ar[u]^(0.7){\textrm{f.t., not}}^{\textrm{necessarily}}^(0.3){\textrm{flat}} & \breve{A}_2 \ar[r] \ar[u]^(0.3){\textrm{f.t.flat}} \ar@{}[ur]|\square & \breve{A}_3 \ar[r] \ar[u]^(0.3){\textrm{f.t.flat}} \ar@{}[ur]|\square & (A^\red)^\sim \ar[u]^(0.3){\textrm{f.t.flat}} \ar[ur] 
} \]
 
Given some $\breve{A_1} \to \breve{B_1}$ extending $A \to B$, consider the sub-$(A^\red)^\sim$-algebra $\breve{B}_1 \subseteq B' \subseteq (B^\red)^\sim$ generated by the image of $\breve{B_1}$. As $(B^\red)^\sim$ is $(A^\red)^\sim$-torsion free, so is $B'$, and as $(A^\red)^\sim$ is a (product of) valuation rings, Lem.\ref{lemm:normHvr}, we deduce that $B'$ is $(A^\red)^\sim$-flat. As $\breve{A}_0 \to \breve{B}_1$ is finite type, $(A^\red)^\sim \to B'$ is also finite type. Since $(A^\red)^\sim$ is a product of local rings, the finitely generated flat module $B'$ is a free module (noetherianness is not needed, \cite[Thm.7.10]{Mat89}). As $(A^\red)^\sim = \bigcup_{\PN(A)} \breve{A}$, we have%
\footnote{\label{foot:descff} This is true for any globally free finite flat morphism $A := \varprojlim A_\lambda \to B$ from the colimit of a filtered system: Choose an $A$-basis $1 = e_1, \dots, e_n$ for $B$ whose first element is the unit. Then the multiplication of $B$ is determined by the $n(n-1)^2$ coefficients of the products $e_ie_j = \sum_k a^{ij}_k e_k \in B \cong A^n$ for $2 \leq i, j \leq n, 1 \leq k \leq n$ subject to linear and quadratic conditions imposed by the commutivity and associativity axioms. These $n(n-1)^2$ elements are in the image of some $A_\lambda$, and the conditions become satisfied in some, possibly higher, $A_{\lambda'}$. They then define a finite free $A_{\lambda'}$-algebra of rank $n$ whose pullback to $A$ is $B$.
%Choose a basis $1 = e_1, \dots, e_n$ for $B' \cong ((A^\red)^\sim)^n$ whose first element is the unit. Then the multiplication of $B'$ is determined by the $n(n-1)^2$ coefficients of the products $e_ie_j \in ((A^\red)^\sim)^n$ for $2 \leq i, j \leq n$ subject to conditions imposed by the commutivity and associativity axioms. Since $(A^\red)^\sim = \bigcup_{\PN(A)} \breve{A}$, these $n(n-1)^2$ elements are contained in some $\breve{A_2} \in \PN(A)$, and we take $\breve{B_2}$ to be the finite free $\breve{A_2}$-algebra that they define.
} %
 $B' \cong (A^\red)^\sim \otimes_{\breve{A_2}} \breve{B_2}$ for some finite free $\breve{A_2} \to \breve{B_2}$ with $\breve{A_2}$ in $\PN(A)$ (the ring $\breve{B_2}$ is not necessarily in $\PN(B)$ yet). Since $(\bigcup_{\breve{A} \supseteq \breve{A_2} \in \PN(A)} \breve{A}) \otimes_{\breve{A_2}} \breve{B_2} \cong B' \supseteq \breve{B}_1$, tensoring with some large enough $\breve{A}_3 \supseteq \breve{A}_2$ in $\PN(A)$, we get our $\breve{B}_3 \cong \breve{A}_3 \otimes_{\breve{A}_2} \breve{B}_2 \supseteq \breve{B}_1 \supseteq \breve{B}_0$, with $\breve{A}_3 \to \breve{B}_3$ in $\PN(\phi)$. %\qedhere
\end{enumerate}
\end{proof}

%\begin{rema}
%In the proof of part~\eqref{lemm:subNormalisation:2}, by also including lifts of generators for the value group extension $\Gamma_{\widetilde{A}} / \Gamma_{R}$ we can also ask that $\Frac(\breve{A}) / \breve{A}^* \to \Frac(\widetilde{A}) / \widetilde{A}^*$ is an isomorphism, or in other words, that $\breve{A} \to \widetilde{A}$ is unramified, cf. \cite[Tag 0ASH]{Stacks}. We don't need this here, but it seems like a useful observation.
%\end{rema}

\begin{lemm} \label{lemm:refRef}
Let $R$ be a hvr of characteristic $p \neq l$. Then any $\fpsl$-morphism $R \to A$ is corefinable by the composition of a generically étale $\fpsl$-morphism $R \to A'$ and a $\uh$-morphism $A' \to A''$ such that $R \to A''$ is also $\fpsl$.
\end{lemm}

\begin{proof}
First consider the case when $A$ is an integral domain. Consider the separable closure $L$ of $\Frac(R)$ in $\Frac(A)$, and choose pseudo-integral closures $A', A''$ of $R, A$ in $L, \Frac(A)$ respectively, Lemma~\ref{lemm:subNormalisation}\eqref{lemm:subNormalisation:Fun}.
\[ \xymatrix@R=12pt{
\Frac(R) \ar[r]^{\textrm{sep.}} & L \ar@<3pt>[r]^{\textrm{purely insep.}} & \Frac(A) \\
R \ar[r]^{\textrm{finite}} \ar[u] \ar[drr] & A' \ar[r]^{\textrm{finite}} \ar[u] & A'' \ar[u] \\
&& A \ar[u] 
} \]
By definition, $A' \to \widetilde{A'}$ and $A'' \to \widetilde{A''}$ induce completely decomposed $\uh$-morphisms, and $\widetilde{A'} \to \widetilde{A''}$ induces a universal homeomorphism by Lem.\ref{lemm:uhClosed}. Hence $A' \to A''$ is a universal homeomorphism.

So now it suffices to show that every $\fpsl$-morphism $R \to A$ is corefinable by an $\fpsl$-morphism $R \to A''$ with $A''$ an integral domain. Suppose first that $R$ is a field $K = R$. Then one of the residue fields of $A$ is of degree prime to $l$: Indeed, write $A = \prod A_{\p_i}$ as the product of its local rings. As $l \nmid \dim_K A$ we have $l \nmid \dim_K A_{\p_i}$ for some $i$. As $\dim_K A_{\p_i} = \sum_j \dim_K  \p_i^j / \p_i^{j+1}$ and each $\p_i^j / \p_i^{j+1}$ is a $k(\p_j)$-vector space, we see that $l \nmid [k(\p_i) : K]$. For a general $R$, the previous case applied to $\Frac(R) \to \Frac(R) \otimes_R A$ produces a minimal prime $\p$ of $A$ such that $l \nmid [A_\p^\red : \Frac(R)]$. But then since flat = torsion-free over valuation rings, $R \to A / \p$ is $\fpsl$.
\end{proof}

%\begin{lemm} \label{lemm:phvrPresFin}
%Suppose that $A$ is a pseudo hvr. Then $A \to \widetilde{A}$ is a filtered colimit of morphisms $A \to A' \to \widetilde{A}$ such that $\Spec(A') \to \Spec(A)$ is a completely decomposed universal homeomorphism of \emph{finite presentation}.
%\end{lemm}
%
%\begin{proof}
%By definition $\Spec(\widetilde{A}) \to \Spec(A)$ is a completely decomposed universal homeomorphism $(\cduh)$. Consider the collection of pairs of \emph{finite} sets $(\{a_i\}_{i \in I}, \{b_j\}_{j \in J})$ where $a_i \in \widetilde{A}$ and $b_j \in \ker(A[x_i : i \in I] \to \widetilde{A}; x_i \mapsto a_i)$. Union induces the set $M$ of pairs $(\{a_i\}, \{b_j\})$ with the structure of a commutative mono{\"i}d. Our indexing category has an object for every $m \in M$, and $\hom(m, m') = \{n \in M : m + n = m'\}$. One checks directly that this category is filtered. %
%% given $m, m'$, the object $m + m'$ admits a morphism from both of them. Given $n, n'$ such that $m + n = m + n' := m'$, the morphism $n + n'$ from $m'$ to itself has $n + n + n' = n' + n + n'$ as a morphism from $m$ to $m'$.
%%
%The object $(\{a_i\}, \{b_j\})$ is sent to $A \to A[x_i] / \langle b_j \rangle \to A[a_i] \to \widetilde{A}$.
%Note $\Spec$ of all of these morphisms are $\cduh$ because $\Spec$ of $A \to \widetilde{A}$ is a $\cduh$.
%\end{proof}

\begin{lemm} \label{lemm:finhvrrefine}
If $R$ is a finite rank hvr and $R \to A$ a finite $R$-algebra, then every $\ldh$-cover of $\Spec(A)$ is refinable by a finite one.
\end{lemm}

\begin{proof}
Since every $\ldh$-covering is refinable by the composition of an $\fpsl$-covering and a $\cdh$-covering (see the proof of \ref{prop:SCH}\eqref{prop:SCH:hvrrefine}) it suffices to prove the statement for $\cdh$-coverings. %
%Since everything is affine, we will say a morphism of rings $A \to B$ is a $\tau$-covering if $\Spec(B) \to \Spec(A)$ is a $\tau$-covering, for $\tau = \ldh, \fpsl$. 
Replacing $A$ with $\prod_{\substack{\p \subset A \\ \p \textrm{ prime}}} (A/\p)^\smallsmile$ we can assume that $A$ is a pseudo-hvr. That is, $\widetilde{A}$ is a hvr and $\Spec(\widetilde{A}) \to \Spec(A)$ is a completely decomposed universal homeomorphism. Let $\Spec(B) \to \Spec(A)$ be the $\cdh$-covering in question, which we assume is affine, and since our $\cdh$-topology is pulled back from a noetherian scheme, cf.Prop.~\ref{prop:inducedSCH}, we also assume its of finite presentation $B \cong A[x_1, \dots, x_n] / \langle f_1, \dots, f_m\rangle$. Every $\cdh$-covering of a hvr admits a section, so there is a factorisation $A \to B \to \widetilde{A}$. As $A \to \widetilde{A}$ is an integral extension, the images of the $x_i$ in $\widetilde{A}$ satisfy some monic polynomials $g_i(T) \in A[T]$. Then the composition $A \to B \to B' := A[x_1, \dots, x_n] \langle f_1, \dots, f_m, g_1(x_1), \dots, g_n(x_n)\rangle$ is a finite morphism, and its $\Spec$ is completely decomposed because $\Spec$ of the composition $A \to B' \to \widetilde{A}$ is completely decomposed.
\end{proof}

\section{$\ldh$-descent for pseudo-hvrs} \label{sec:hvrldh}

From this point we start working with the following ``Gestern'' condition. A presheaf $F$ satisfies (G1) if:
\begin{enumerate}
 \item[(G1)] For every hvr $R$, the canonical morphism $F(R) \to F(\Frac(R))$ is a monomorphism.
\end{enumerate}

\begin{rema} \label{rema:whyGersten}
We find it disappointing that we do not know a proof avoiding this condition, as its not really clear heuristically why it should be involved in passing traces from $F$ to $F_\ldh$.

Voevodsky shows in \cite[Cor.4.18]{Voe00b} that for a homotopy invariant presheaf with transfers $F$ and a smooth semi-local $k$-scheme $X$, the morphism $F(X) \to F(\eta)$ to the generic scheme $\eta$ is a monomorphism. So, at least in the homotopy invariant setting over a field, traces imply a version of (G1). However, in our setting we do not have this: if $S$ is a noetherian base scheme of dimension $> 0$, $s$ a non-generic point, and $F$ the constant additive sheaf of some nontrivial abelian group $A$, then $F(-\times_S s)$ is clearly $\uh$-invariant, and has traces by Example~\ref{exam:traces}\eqref{exam:traces:const}, however, will not satisfy (G1).
\end{rema}

%On the other hand, Voevodsky shows in \cite[Prop.4.24]{Vctpt} that for a homotopy invariant presheaf with transfers, $F_\Zar(W) = F(W)$ for regular semi-local $k$-schemes. Compare this to Prop.\ref{prop:FisFldhonHvr} 

\begin{prop} \label{prop:FisFldhonHvr}
Suppose that $F$ is a $\uh$-invariant $\zll$-linear presheaf with traces satisfying (G1). Then for any pseudo-hvr $R$ of positive characteristic $p \neq l$, we have $F(R) = F_\ldh(R)$.
\end{prop}

\begin{proof}
Since $F$ is $\uh$-invariant, so is $F_\ldh$, Cor.\ref{coro:Huhinv}, so we can assume $R$ is a hvr.

Injectivity: $s \in \ker \bigl ( F(R) {\to} F_\ldh(R) \bigr )$ if and only if there is some $\ldh$-covering $f: Y \to \Spec(R)$ such that $F(f)s = 0$. Since $R$ is a hvr we can assume that $f$ is an $\fpsl$-morphism, Prop.\ref{prop:SCH}(4). But then $s \stackrel{(Deg)}{=} \tfrac{1}{\deg f} \Tr_f F(f) s = 0$.

Surjectivity: For every $s \in F_\ldh(R)$ there is an $\ldh$-covering $f: Y \to \Spec(R) =: X$ such that $s|_{Y} \in \operatorname{im} (F(Y) {\to} F_\ldh(Y))$. Since $R$ is a hvr we can assume that $f$ is an $\fpsl$-morphism, Prop.\ref{prop:SCH}\eqref{prop:SCH:hvrrefine}. In fact, we can assume $f$ factors as $Z_0 \stackrel{f_1}{\to} Y_0 \stackrel{f_0}{\to} X$ with $f_1$ a $\uh$-morphism, and $f_0$ a generically étale $\fpsl$-morphism, Lem.\ref{lemm:refRef}. Since $F$ is $\uh$-invariant, so is $F_\ldh$, Cor.\ref{coro:Huhinv}, so we can forget $f_1$ and just work with $f_0$. Since $f_0$ is generically étale, $Y_0 \times_X Y_0$ is generically reduced. Choose $Y_1 = (Y_0 \times_X Y_0)^\smallsmile$ such that the composition $\pi_1: Y_1 \to Y_0 \times_X Y_0 \stackrel{pr_1}{\to} Y_0$ is still flat, Lem.~\ref{lemm:subNormalisation}\eqref{lemm:subNormalisation:flat}, and therefore $\fpsl$. We claim that
\[ 0 \to F(X) \stackrel{F(f_0)}{\longrightarrow} F(Y_0) \stackrel{F(\pi_1) - F(\pi_2)}{\longrightarrow} F(Y_1) \]
is exact where $\pi_2$ is the composition $Y_1 \to Y_0 \times_X Y_0 \stackrel{pr_2}{\to} Y_0$. Indeed, by $\id \stackrel{(Deg)}{=} \tfrac{1}{\deg f_0} \Tr_{f_0}F(f_0)$ it is exact at $F(X)$. For exactness at $F(Y_0)$, we claim that $\Tr_{\pi_1} F(\pi_2) = F(f_0) \Tr_{f_0}$. Indeed, since $Y_1$ is $\Spec$ of a (product of) psuedo-hvrs, and $F$ is $\uh$-invariant, by (G1) it suffices to check this after pulling back to the generic points of $Y_1$. But by (CdB), it suffices to show $\Tr_{\pi_1} F(\pi_2) = F(f_0) \Tr_{f_0}$ holds generically, i.e., over the generic point $\eta$ of $X$. But 
\[ \xymatrix{
\eta \times_X Y_1 \ar[r] \ar[d] & \eta \times_X Y_0 \ar[d] \\
\eta \times_X Y_0 \ar[r] & \eta
} \]
is cartesian, so the claim follows from (CdB). Hence, the above sequence is exact at $F(Y_0)$ since if $F(\pi_2)s' = F(\pi_1)s'$, then $s' 
\stackrel{\textrm{(Deg)}}{=} \tfrac{1}{\deg \pi_1} \Tr_{\pi_1} F(\pi_1)s'
\stackrel{\textrm{cycle}}{=} \tfrac{1}{\deg \pi_1} \Tr_{\pi_1} F(\pi_2)s'
\stackrel{\textrm{(``CdB'')}}{=} \tfrac{1}{\deg \pi_1} F(f_0) \Tr_{f_0}s'$.

So we have extablished that the top row in the following diagram is exact. Injectivity for pseudo-hvrs says the right vertical morphism is injective. Hence, by diagram chase, we find a preimage of $s$ in $F(X)$.
\[ \xymatrix@R=12pt{
0 \ar[r] & F(X) \ar[r] \ar[d] & F(Y_0) \ar[r] \ar[d] & F(Y_1) \ar[d] \\
0 \ar[r] & F_\ldh(X) \ar[r] & F_\ldh(Y_0) \ar[r] & F_\ldh(Y_1)
} \]
%The sheafification can be calculated by applying the \v{C}ech cohomology functor twice \cite[Exp.II.Prop.3.2,Rem.3.3]{SGA4} so it suffices {\color{red} maybe false, since sheafifiaction moves outside hvrs Instead use that $F$ is already separated?} to show that $F(R) = \check{H}^0_\ldh(R, F)$. Every $\ldh$-covering is refinable by an $\fpsl$-covering, Prop.\ref{prop:SCH}\eqref{prop:SCH:hvrrefine}. So $\check{H}^0_\ldh(R, F) = \check{H}^0_\fpsl(R, F)$, and it suffices to show that $F(R) = \check{H}^0_\fpsl(R, F)$. But $F$ is automatically an $\fpsl$-sheaf by Lem.\ref{lemm:hcVanish}.
\end{proof}

\section{Traces on $F_\ldh$} \label{sec:tracesFldh}

In this section we show that for a nice presheaf with traces $F$, the associated $\ldh$-sheaf also has traces. We essentially transplant the method of \cite[Thm.3.5.5]{Kel12}, which becomes much easier in the context of this paper. For another approach to putting trace morphisms on $F_\ldh$ (with more restrictive hypotheses than we use here) see \cite{Kel17}.

Recall that in \cite[\S 3.5]{Kel12} we defined 
\[ F_\cdd(X) = \prod_{x \in X} F(x) \]
for $F$ a presheaf and $X$ a scheme.

\begin{rema} \label{rema:cddbasic}
%\begin{enumerate}
% \item \label{rema:cddbasic:cdd} 
One sees directly that $F_\cdd$ is a $\cdh$-sheaf. More specifically, for any completely decomposed morphism $Y \to X$, the sequence
\[ 0 \to F_\cdd(X) \to F_\cdd(Y) \to F_\cdd(Y \times_X Y) \]
is exact (showing this by hand using a splitting of $\amalg_{y \in Y}y \to \amalg_{x \in X}x$ is an easy exercise).
%
% \item One also sees directly that if $F$ is $\uh$-invariant, then so is $F_\cdd$.
%\end{enumerate}
\end{rema}

What takes a little bit more work is the following theorem.

\begin{prop} \label{prop:cddTraces} {\cite[Thm.3.5.5, Lem.3.3.6(2), Def.3.3.4, Prop.3.5.7]{Kel12}}
Let $F$ be a $\uh$-invariant $\zll$-linear presheaf with traces, where $l \neq \operatorname{char}s$ for all points $s \in S$ of the base scheme. Then there is a unique structure of traces on $F_\cdd$ such that $F \to F_\cdd$ is a morphism of presheaves with traces. Moreover, the trace morphisms on $F_\cdd$ satisfy:

\begin{enumerate}
 \item[(Tr)] If $A$ is a psuedo-hvr, $\phi: A \to B$ an $\fps$ morphism, $\p_1, \dots, \p_n$ are the minimial ideals of $B$, and $\eta_i: B \to (B / \p_i)^\smallsmile$ the canonical morphisms, then 
\[ \xymatrix@R=0pt{
& 
F(B) \ar[rr]^{\sum F(\eta_i)} \ar[ddr]_{\Tr_\phi} && 
\oplus F((B / \p_i)^\smallsmile) \ar[ddl]^{\sum m_i \Tr_{\eta_i \circ \phi}} \\
\Tr_\phi = \sum m_i \Tr_{\eta_i \circ \phi} F(\eta_i) &&& \\
&& 
F(A) & 
} \]
where $m_i = \length\ B_{\p_i}$, and of course, $(B / \p_i)^\smallsmile$ are chosen to be flat over $A$, (or some $\cduh$-extension of it, and we implicitly use $\uh$-invariance), Lem.\ref{lemm:subNormalisation}\eqref{lemm:subNormalisation:flat}.\end{enumerate}
\end{prop}

\begin{rema}
We do not need the following description, but in case the reader is interested, we recall that the trace morphisms on $F_\cdd$ are defined as follows. Given an $\fps$ morphism $f: Y \to X$, and $y \in Y$, the $y$th component of the trace morphism $F_\cdd(Y) = \prod_{y \in Y} F(y) \to \prod_{x \in X} F(x) = F_\cdd(X)$ is given by $(\length\ \OO_{x \times_X Y, y}) \Tr_{f|_{y/x}}$.
\end{rema}

\begin{lemm} \label{lemm:ldhcdd}
Suppose that $F$ is a $\uh$-invariant $\zll$-linear presheaf with traces satisfying (G1), and $R$ a finite rank hvr of positive characteristic $p \neq l$. Then on the category of finite $R$-algebras, there is a factorisation $F \to F_\ldh \stackrel{\iota}{\hookrightarrow} F_\cdd$ with $\iota$ a monomorphism.
\end{lemm}

\begin{proof}
On hvrs, and in particular on fields, $F \cong F_\ldh$, Prop.\ref{prop:FisFldhonHvr}, so $F_\cdd \cong (F_\ldh)_\cdd$. Our injection is $\iota: F_\ldh \to (F_\ldh)_\cdd \stackrel{\sim}{\leftarrow} F_\cdd$.

Let $A$ be a finite $R$-algebra. Spec of the morphism $A \to \prod_{\substack{\p \subset A \\ \p\ \textrm{ prime}}} (A / \p)^\smallsmile$ is a $\cdh$-covering. Since $F_\ldh$ is $\uh$-invariant, Cor.\ref{coro:Huhinv}, each of the morphisms $F_\ldh((A / \p)^\smallsmile) \to F_\ldh((A / \p)^\sim)$ is an isomorphism. But $F \cong F_\ldh$ on the hvrs $(A / \p)^\sim$ and $k(\p)$, Prop.\ref{prop:ldhTraces}, and $F$ satisfies (G1), so each $F_\ldh((A / \p)^\sim) \to F_\ldh(k(\p))$ is injective. Hence, $F_\ldh(A) \to \prod_{\substack{\p \subset A \\ \p\ \textrm{ prime}}} F_\ldh(k(\p)) = (F_\ldh)_\cdd(A) \cong F_\cdd(A)$ is injective.
\end{proof}

\begin{prop} \label{prop:ldhTraces}
Suppose that $F$ is a $\uh$-invariant $\Zll$-linear presheaf with traces satisfying (G1), and $R$ is a finite rank hvr of positive characteristic $p \neq l$. Then on the category of finite $R$-algebras, the trace morphisms on $F_\cdd$ descend to the subpresheaf $F_\ldh$. In particular, $F_\ldh$ has a structure of traces on the category of finite $R$-algebras.
\end{prop}

\begin{proof}
Explicitly, we want to show that for every $\fps$-morphism of $R$-algebras $\phi: A \to B$, the trace morphism $\Tr_\phi: F_\cdd(B) \to F_\cdd(A)$ sends the image of $F_\ldh(B) \hookrightarrow F_\cdd(B)$ inside the image of $F_\ldh(A) \hookrightarrow F_\cdd(A)$. First note that if $A \to A'$ is a morphism whose $\Spec$ is a completely decomposed proper (=finite) morphism, then a diagram chase using the short exact sequences, Rem.\ref{rema:cddbasic}, %, Prop.\ref{prop:SCH}\eqref{prop:SCH:bu}, 
associated to $A \to A' \rightrightarrows A' \otimes_A A'$ by $F_\ldh \hookrightarrow F_\cdd$ shows that $F_\ldh(A) = F_\ldh(A') \cap F_\cdd(A)$. So replacing $A$ with $A' = \prod_{\substack{\p \subset A \\ \p\ \textrm{ prime}}} (A / \p)^\smallsmile$ and using (Add) and (CdB), we can assume that $A$ is a pseudo-hvr. Now, using the morphism $B \to \prod_{\substack{\q \subset B \\ \q\ \textrm{ prime}}} (B / \q)^\smallsmile$ and the property (Tr) stated in Prop.\ref{prop:cddTraces}, we can assume that $B$ is also a pseudo-hvr. So now $A \to B$ is a $\fps$-morphism between pseudo-hvrs. But then $F(A) \stackrel{\sim}{\to} F_\ldh(A)$, $F(B) \stackrel{\sim}{\to} F_\ldh(B)$ are isomorphisms, Prop.\ref{prop:FisFldhonHvr}. So the result follows from the fact that $F \to F_\cdd$ is a morphism of presheaves with traces, Prop.\ref{prop:cddTraces}.
\end{proof}

\section{Comparison of $\cdh$- and $\ldh$-descent} \label{sec:cdhldhComp}

\begin{theo} \label{theo:mainTheorem}
Suppose $S$ is a finite dimensional noetherian separated scheme of positive characteristic $p \neq l$ and $F$ is a $\uh$-invariant $\zll$-linear presheaf with traces satisfying \emph{both} conditions:
\begin{enumerate}
 \item[(G1)] $F(R) \to F(\Frac(R))$ is injective for every finite rank hvr $R$.

 \item[(G2)] $F(R) \to F(R / \p)$ is surjective for every finite rank hvr $R$ and prime $\p$ of codimension one.
\end{enumerate}
Then the canonical comparison morphism is an isomorphism:
\[ H_\cdh^n(S, F_\cdh) \stackrel{\sim}{\to} H_\ldh^n(S, F_\ldh). \]
\end{theo}

\begin{proof}
By the change of topology spectral sequence
\[ H^i_\cdh(S, (\ulh_\ldh^jF)_\cdh) \implies H^{i+j}_\ldh(S, F_\ldh) \]
it suffices to show that $F_\cdh = F_\ldh$, and $(\ulh_\ldh^jF)_{\cdh} = 0$ for $j > 0$. Since finite rank hvrs form a conservative family of fibre functors for the $\cdh$-site $\Sch_S$, Cor.\ref{coro:finiteRankConservative}, and cohomology commutes with filtered limits of schemes with affine transition morphisms, Prop.\ref{prop:SCH}, it suffices to show that for every finite rank hvr $R$ we have $F(R) = F_\ldh(R)$, and $H^j_\ldh(R, F_\ldh) = 0$ for $j > 0$.

It was already shown in Prop.\ref{prop:FisFldhonHvr} that we have $F(R) = F_\ldh(R)$. We now show that for any finite $R$-algebra $A$, we have \begin{equation} \label{equa:toshow}
H^j_\ldh(A, F_\ldh) = 0%; \qquad \textrm{  for } j > 0. 
\end{equation}
for $j > 0$. We work by induction on $(\dim \Spec(A), j)$ where $\NN \times \NN_{>0}$ has the lexicographical ordering. Explicitly, we suppose that \eqref{equa:toshow} is true for $\dim \Spec(A) < \dim \Spec(R)$, and all $0 < j$ and suppose also that \eqref{equa:toshow} is true when $\dim \Spec(A) = \dim \Spec(R)$ and $0 < j < J$. We will show that it is true for $\dim \Spec(A) = \dim \Spec(R)$ and $j = J$.

Here is a plan of what we will prove, where $A' = A^\pic$ is a pseudo-integral closure of $A^\red$ in $(Q(R) \otimes_R A)^\red$, and $R' = \widetilde{A'}$:
\begin{align*}
H_\ldh^J(A, F_\ldh)
&\stackrel{(\alpha)}{\cong} H_\ldh^J(A', F_\ldh) \qquad \textrm{blowup l.e.s, induction via Eq.\eqref{eq:bup}, $\uh$-inv., (G2)} \\
&\stackrel{(\beta)}{\cong} H_\ldh^J(R', F_\ldh) \qquad \uh\textrm{-inv., Cor.\ref{coro:Huhinv}} \\
&\stackrel{(\gamma)}{\cong} \check{H}_\ldh^J(R', F_\ldh) \qquad \textrm{induction via Eq.\eqref{equa:checkCoh}} \\
&\stackrel{(\delta)}{\cong} \check{H}_\fpsl^J(R', F_\ldh) \qquad R' \textrm{is a product of hvrs, Prop.\ref{prop:SCH}\eqref{prop:SCH:hvrrefine}} \\
&= 0 \qquad \qquad \qquad \qquad \textrm{traces, Lem.\ref{lemm:hcVanish}, Prop.\ref{prop:ldhTraces}}
\end{align*}

\emph{Step $\alpha$}. Let $A' = A^\pic$ be a pseudo-integral closure of $A^\red$ inside $(Q(R) \otimes_R A)^\red$, let $\p \subset R$ be the prime of height one (if $\dim \Spec(R) = 0$ set $\p$ to be the unit ideal $\p = R$), and consider the blowup sequence, Prop.\ref{prop:SCH}\eqref{prop:SCH:bu},
\begin{align} \label{eq:bup}
\dots 
&\to H_\ldh^{j-1}(A' / \p, F_\ldh)
\\&\to H_\ldh^j(A, F_\ldh) 
\to H_\ldh^j(A', F_\ldh) \oplus H_\ldh^j(A / \p, F_\ldh)
\to H_\ldh^j(A' / \p, F_\ldh)
\to \dots \notag
\end{align}
By induction on $\dim \Spec(A)$, \eqref{equa:toshow} is true for $A / \p$ and $A' / \p$, so we obtain an isomorphism 
\begin{equation} \label{equa:ldhAAbreve}
H_\ldh^J(A, F_\ldh) \cong H_\ldh^J(A', F_\ldh).
\end{equation}
Here, (G2) and $\uh$-invariance of $F_\ldh$, Cor.\ref{coro:Huhinv}, is used in the case $J = 1$ to obtain surjectivity of the morphism $F_\ldh(A') \to F_\ldh(A' / \p)$.

\emph{Step $\beta$}. Note $\Spec$ of $A' \to R' := \widetilde{A'}$ is a $\uh$. Since $H_\ldh^J(-, F_\ldh)$ is $\uh$-invariant, Cor.\ref{coro:Huhinv}, we are reduced to showing that 
%the isomorphic (cf. Eq.~\eqref{equa:checkCoh}) groups 
$H_\ldh^J(R', F_\ldh)$ % \cong H_\ldh^J(R', F_\ldh)$ 
vanishes.% whenever $R'$ is a product of hvrs. 

\emph{Step $\gamma$}. Consider the \v{C}ech spectral sequence 
\begin{equation} \label{equa:SS}
E_2^{i,j} = \check{H}^i_\ldh(R', \ulh_\ldh^jF) \implies H^{i+j}_\ldh(R', F_\ldh). 
\end{equation}
Note that the $i = 0, j > 0$ part vanishes automatically, since $\check{H}^0_\tau(-, \ulh_\tau^jF) = 0$ vanishes for any topology $\tau$ and $j > 0$. In particular, 
\begin{equation} \label{equa:firstVan}
E_2^{0, J} = \check{H}^0_\ldh(R', \ulh_\ldh^JF) \cong 0.
\end{equation}
Since every $\ldh$-covering of $R'$ is refinable by a finite one, Lem.~\ref{lemm:finhvrrefine}, and $\underline{H}_\ldh^jF(-) = H_\ldh^j(-, F_\ldh)$ vanishes on finite $R'$-algebras for $0 < j < J$ by induction, it follows that
\begin{equation} \label{equa:secondVan}
E_2^{J-j,j} = \check{H}^{J-j}_\ldh(R', \ulh_\ldh^jF) \cong 0; \qquad \textrm{  for } 0 < j < J. 
\end{equation}
The vanishing so far, \eqref{equa:firstVan} and \eqref{equa:secondVan}, with the spectral sequence \eqref{equa:SS} shows that
\begin{equation} \label{equa:checkCoh}
E_2^{J, 0} = \check{H}_\ldh^J(R', F_\ldh) \cong H_\ldh^J(R', F_\ldh).
\end{equation}
%Let $A' = A^\pic$ be a pseudo-integral closure of $A^\red$ in $(Q(R) \otimes_R A)^\red$, let $\p \subset R$ be the prime of height one (if $\dim \Spec(R) = 0$ set $\p$ to be the unit ideal $\p = R$), and consider the blowup sequence, Prop.\ref{prop:SCH}\eqref{prop:SCH:bu},
%\begin{align} \label{eq:bup}
%\dots 
%&\to H_\ldh^{j-1}(A' / \p, F_\ldh)
%\\&\to H_\ldh^j(A, F_\ldh) 
%\to H_\ldh^j(A', F_\ldh) \oplus H_\ldh^j(A / \p, F_\ldh)
%\to H_\ldh^j(A' / \p, F_\ldh)
%\to \dots \notag
%\end{align}
%By induction on $\dim \Spec(A)$, \eqref{equa:toshow} is true for $A / \p$ and $A' / \p$, so we obtain an isomorphism 
%\begin{equation} \label{equa:ldhAAbreve}
%H_\ldh^J(A, F_\ldh) \cong H_\ldh^J(A', F_\ldh).
%\end{equation}
%Here, (G2) and $\uh$-invariance of $F_\ldh$, Cor.\ref{coro:Huhinv}, is used in the case $J = 1$ to obtain surjectivity of the morphism $F_\ldh(A') \to F_\ldh(A' / \p)$.

\emph{Step $\delta$}. Since $R'$ is a product of hvrs, every $\ldh$-covering is refinable by an $\fpsl'$-covering, Prop.\ref{prop:SCH}\eqref{prop:SCH:hvrrefine}. So it suffices to show that the isomorphic groups
\begin{equation} \label{equa:ldhFpsl}
\check{H}_\ldh^J(R', F_\ldh) \cong \check{H}_\fpsl^J(R', F_\ldh)
\end{equation}
are zero.

\emph{Step $\epsilon$}. Since $F_\ldh$ has a structure of traces, Prop.\ref{prop:ldhTraces}, this vanishing follows directly from $\zll$-linearity and the structure of traces, Lem.\ref{lemm:hcVanish}. 
%To recap, 
%\begin{align*}
%0 
%&\cong \check{H}_\fpsl^J(R', F_\ldh) \qquad \textrm{traces, Lem.\ref{lemm:hcVanish}, Prop.\ref{prop:ldhTraces}} \\
%&\cong \check{H}_\ldh^J(R', F_\ldh) \qquad R' \textrm{is an hvr, Prop.\ref{prop:SCH}\eqref{prop:SCH:hvrrefine}} \\
%&\cong H_\ldh^J(R', F_\ldh) \qquad %E_2^{<J, j} = 0, \textrm{ Eq.\eqref{equa:SS}, \eqref{equa:firstVan}, \eqref{equa:secondVan}} 
%\textrm{induction via Eq.\eqref{equa:checkCoh}} \\
%&\cong H_\ldh^J(A, F_\ldh) \qquad \uh\textrm{-inv., Cor.\ref{coro:Huhinv}, blowup l.e.s, Eq.\eqref{eq:bup}, Prop.\ref{prop:SCH}\eqref{prop:SCH:bu}} \\
%&\cong H_\ldh^J(A', F_\ldh) \qquad \textrm{induction via Eq.\eqref{equa:checkCoh}} \\
%&\cong H_\ldh^J(A, F_\ldh) \qquad \textrm{blowup l.e.s, Prop.\ref{prop:SCH}\eqref{prop:SCH:bu}, (G2), induction via Eq.\eqref{eq:bup}}
%\end{align*}
\end{proof}

\section{Appendix A. The $\cdh$- and $\ldh$-topologies} \label{sec:cdhldh}

In \cite{GK15} it was observed that hensel valuation rings, or hvrs, form a conservative family of fibre functors for the $\cdh$-site of a noetherian scheme (see Section~\ref{sec:hvr} for some facts about hvrs). Here, we observe that if $\dim S$ is finite, then in fact, it suffices to consider hvrs of finite rank.

\begin{defi} 
Let $S$ be a separated noetherian scheme, $\Sch_S$ the category of separated finite type $S$-schemes, and $l \in \ZZ$ a prime. We quickly recall the following definitions.
\begin{enumerate}
 \item A morphism $f: Y \to X$ is \emph{completely decomposed} if for all $x \in X$ there exists $y \in Y$ with $f(y) = x$ and $k(y) = k(x)$.
 \item The $\cdh$-topology is generated by families of étale morphisms $\{Y_i \to X\}_{i \in I}$ such that $\amalg Y_i \to X$ is completely decomposed, and families of proper morphisms $\{Y_i \to X\}_{i \in I}$ such that $\amalg Y_i \to X$ is completely decomposed.
 
 \item The $\ldh$-topology is generated by the $\cdh$-topology, and finite flat surjective morphisms of degree prime to $l$.
\end{enumerate}
\end{defi}

\begin{lemm} \label{lemm:hvrFinOK}
Let $A$ be a noetherian ring, $R$ a hvr, and $A \to R$ a morphism. Then $R$ is a filtered colimit $A$-algebras which are finite rank hvrs.
\end{lemm}

\begin{proof}
Note $A \to R$ factors through a localisation of $A$, and local noetherian rings have finite Krull dimension, so we can assume $A$ is has finite Krull dimension. Certainly, $R$ is the filtered union of its finitely generated sub-$A$-algebras. For each such $A$-algebra $A \to A_\lambda \subseteq R$, define $R_\lambda = \Frac(A_\lambda) \cap R \subseteq \Frac(R)$ to be the valuation ring induced on the fraction field of $A_\lambda$ by $R$. This is finite rank: Certainly, $R_\lambda$ is the union of its finitely generated sub-$A_\lambda$-algebras $A_\lambda \subseteq A_{\lambda \mu} \subseteq R_\lambda$. By $A_\lambda \subseteq A_{\lambda \mu} \subseteq \Frac(A_\lambda)$ we have $\Frac(A_\lambda) = \Frac(A_{\lambda \mu})$ so $\dim A_{\lambda} \geq \dim A_{\lambda \mu}$ \cite[Thm.5.5.8]{EGAIV2}, and therefore%
\footnote{Suppose that $\p_0 \supsetneq \dots \supsetneq \p_n$ is a sequence of prime ideals of $R_\lambda$ with $n > \dim A_\lambda$. For each $i$ choose $a_i \in \p_i \setminus \p_{i +1}$, and consider the finitely generated sub-$A_\lambda$-algebra $A'_\lambda = A_\lambda[a_0, \dots, a_n] \subseteq R_\lambda$. Then $\p_i \cap A'_\lambda \neq \p_{i+1} \cap A'_\lambda$ for each $i$, but by  \cite[Thm.5.5.8]{EGAIV2} we have $\dim A'_\lambda \leq \dim A_\lambda$, so there is a contradiction and we conclude that $\p_0 \supsetneq \dots \supsetneq \p_n$ cannot exist.
} %
$\dim \Spec(A_\lambda) \geq \dim \Spec(R_\lambda)$. 

Consider the henselisations $R_\lambda^h$ of the $R_\lambda$. Henselisations of valuation rings are valuation rings of the same rank, \cite[Tag 0ASK]{Stacks}, so the $R_\lambda^h$ are also of finite rank. The inclusion $R_\lambda \subseteq R$ extends uniquely to an inclusion%
\footnote{The map $R_\lambda^h \subseteq R$ is indeed injective: Henselisations of valuation rings are valuation rings of the same rank, \cite[Tag 0ASK]{Stacks}. In particular, $R_\lambda^h$ is an integral domain, and $\Spec(R_\lambda^h) \to \Spec(R_\lambda)$ is an isomorphism of topological spaces. As $\Spec(R) \to \Spec(R_\lambda)$ sends the generic point to the generic point, $\Spec(R) \to \Spec(R_\lambda^h)$ must also send the generic point to the generic point. In other words, $R_\lambda^h \to R$ is injective.} %
$R_\lambda^h \subseteq R$ as $R_\lambda \subseteq R$ local morphism of local rings (indeed $R_\lambda^* = R^*$) towards a hensel local ring. Moreover, any inclusion of finitely generated sub-$A$-algebras $A_\lambda \subseteq A_{\lambda'} \subseteq R$ induces a unique factorisation $R_\lambda^h \subseteq R_{\lambda'}^h \subseteq R$ for the same reason. For every $a \in R$, there is a finitely generated sub-$A$-algebra $A_\lambda$ with $a \in A_\lambda$. Clearly, this implies $a \in R_\lambda$, so $a \in R_\lambda^h$, and it follows that $R$ is the union of the finite rank hvrs $R_\lambda$, and this is a filtered union because the poset $\{A_\lambda\}$ is filtered.
\end{proof}

\begin{coro} \label{coro:finiteRankConservative}
Let $S$ be a noetherian separated scheme, and $\Sch_S$ the category of finite type separated $S$-schemes equipped with the $\cdh$-topology. For any $S$-scheme $P \to S$ define 
\[ F(P) = \varinjlim_{P \to X \to S} F(X), \]
where the colimit is over factorisations with $X \to S$ in $\Sch_S$. Then the family of functors
\[ \left \{ \begin{array}{ccc}
\Shv_\cdh(\Sch_S) &\to& Set \\
F &\mapsto& F(P)
\end{array}
\middle  | \begin{array}{c}
\Spec(R) \to S \\
R \textrm{ is a finite rank hvr} \end{array} \right \} \]
is a conservative family of fibre functors. 
\end{coro}

\begin{proof}
It was proven in \cite[Thm.2.3, Thm.2.6]{GK15} that the family of all hvrs induces a conservative family of fibre functors. But Lemma~\ref{lemm:hvrFinOK} says that any hvr is a filtered colimit of finite rank hvrs. So given a cdh-sheaf $F$, if $F(R) = 0$ for every finite rank hvr, we have $F(R) = 0$ for all hvrs, and therefore $F = 0$.
\end{proof}

%\appendix 
\section{Appendix B. Sites of nonnoetherian schemes} \label{sec:nonnoeth}

\begin{defi}
We write $\SCH_S$ for the category of \emph{all}%
\footnote{This is a bit of overkill, since we just really only want to enlarge $\Sch_S$ to include schemes of the form $\Spec(A) {\to} \Spec(R) {\to} S$ with $R$ a finite rank valuation ring and $R \to A$ finite, but whatever. } %
 quasi-compact separated (and therefore quasi-separated) $S$-schemes.
\end{defi}
 
\begin{rema} \label{rema:SCH}
Since our base scheme $S$ will always be a noetherian separated scheme, and in particular quasi-compact quasi-separated, $\SCH_S$ is nothing more than the category of those $S$-schemes $T \to S$ of the form $T = \varprojlim_{\lambda \in \Lambda} T_\lambda$ for some filtered system $T_-: \Lambda \to \Sch_S$ with affine transition morphisms, \cite[Thm.1.1.2]{Tem11}.
\end{rema}

\begin{rema}
Following Suslin and Voevodsky, we use the term \emph{covering family} in the sense of \cite[Exp.II.Def.1.2]{SGA4}. That is, in addition to satisfying the axioms of a pretopology, any family refinable by a covering family is a covering family.
\end{rema}

\begin{prop} \label{prop:inducedSCH}
Let $S$ be a noetherian separated scheme, let $\tau$ be a topology on $\Sch_S$ such that every covering family is refinable by one indexed by a finite set, and let $\tau'$ be the coarsest topology on $\SCH_S$ making $\Sch_S \to \SCH_S$ continuous, cf. \cite[Exp.III.Prop.1.6]{SGA4}. Then the covering families for $\tau'$ are those families which are refinable by pullbacks of covering families in $\Sch_S$.
\end{prop}

\begin{proof}
Certainly any $\tau$-covering family in $\Sch_S$ must be a $\tau'$-covering family in $\SCH_S$, and therefore the pullback of any $\tau$-covering family in $\Sch_S$ must also be a $\tau'$-covering family in $\SCH_S$, so it suffices to show that the collection of such families (i) contains the identity family, (ii) is closed under pullback, and (iii) is closed under ``composition'' in the sense that if $\{U'_i \to X'\}_{i \in I}$ and $\{V'_{ij} \to U'_i\}_{j \in J_i}$ are such families, then so is $\{V'_{ij} \to X'\}_{i \in I, i \in J_i}$. The first two are clear, so consider the third. By hypothesis, without loss of generality we can assume that $I$ is finite. Suppose $U'_i \to Y_i$, and $X' \to X$ are morphisms with $Y_i, X \in \Sch_S$, and $\{V_{ij} \to Y_i\}, \{U_i \to X\}$ are $\tau$-coverings such that $U_i' = U_i \times_X X'$ and $V_{ij}' = V_{ij} \times_{Y_i} U'_i$. 
\[ \xymatrix@!=6pt{
V_{ij}' \ar[r] \ar@/_6pt/[dd] & U_i' \ar[r] \ar@/_6pt/[dd] \ar[d] & X' \ar[d] & \in \SCH_S\\
&  U_i \ar[r] & X & \in \Sch_S \\
V_{ij} \ar[r] & Y_i 
} \]
%\[ \xymatrix@!=0pt{
%V_{ij} \ar[rr] \ar[d] && U_i \ar[rr] \ar[dl] \ar[dr] && X \ar[d] && \in \SCH_S\\
%V_{ij}' \ar[r] & Y_i && U_i' \ar[r] & X' && \in \Sch_S
%} \]
Without loss of generality we can assume that $X'$ is the limit $\varprojlim_{\Lambda} X_\lambda$ of a filtered system $\{X_\lambda\}$ in $\Sch_S$ with affine transition morphisms, cf. Rem.\ref{rema:SCH}, and since $\hom(\varprojlim X_\lambda, X) = \varinjlim \hom(X_\lambda, X)$ \cite[Cor.8.13.2]{EGAIV3} that $X = X_{\lambda_0}$ for some $\lambda_0 \in \Lambda$. In particular, now $U_i' = \varprojlim_{\lambda \leq \lambda_0} (X_\lambda \times_{X_{\lambda_0}} U_i)$. Now since $\hom(\varprojlim_{\lambda \leq \lambda_0} (X_\lambda \times_{X_{\lambda_0}} U_i), Y_i) = \varinjlim_{\lambda \leq \lambda_0} \hom(X_\lambda \times_{X_{\lambda_0}} U_i, Y_i)$ \cite[Cor.8.13.2]{EGAIV3}, for each $i$ we can assume that $Y_i = X_{\lambda_i} \times_{X_{\lambda_0}} U_i$ for some $\lambda_i \leq \lambda_0$. Choosing a $\mu \leq \lambda_i$ small enough (this is where we use finiteness of $I$) and pulling back everything to $X_\mu$, we can assume that $Y_i = U_i$. In this case, $\{V'_{ij} \to U'_i \to X'\}$ is the pullback of the $\tau$-covering family $\{V_{ij} \to U_i \to X_\mu\}$ in $\Sch_S$.
\end{proof}

In light of Proposition~\ref{prop:inducedSCH}, the $\tau'$-covers in $\Sch_S$ are refinable by the $\tau$-covers, so for such topologies (e.g., $\cdh, \ldh, \fpsl$) we use the same symbol to denote the induced topology on $\SCH_S$. Another consequence of this observation is that the adjunction
\[ \iota^s: \Shv_\tau(\Sch_S) \rightleftarrows \Shv_\tau(\SCH_S): \iota_s \]
induced by the continuous functor $\Sch_S {\to} \SCH_S$ satisfies $\iota_s\iota^s = \id$. See \cite[Expos{\'e} 3]{SGA4} for some material about this basic adjunction. We will write 
\[ \iota^*: \PreShv(\Sch_S) \rightleftarrows \PreShv(\SCH_S): \iota_* \]
for the presheaf adjunction.

%Pro-objects: 
%\[ \hom(\varprojlim X_\lambda, \varprojlim Y_\mu) = \varprojlim_\mu \varinjlim_\lambda \hom(X_\lambda, Y_\mu). \]

\begin{prop} \label{prop:SCH}
Let $S$ be a noetherian separated scheme, and $T {\to} S$ in $\SCH_S$. Write $T$ as the limit 
\[ T = \varprojlim_{\lambda \in \Lambda} T_\lambda \]
of some filtered system $\{T_\lambda\}$ in $\Sch_S$ with affine transition morphisms, cf.~Remark~\ref{rema:SCH}. Let $\tau$ be both a topology on $\Sch_S$ such that every covering family is refinable by one indexed by a finite set, and also the induced topology on $\SCH_S$, e.g., $\tau = \cdh, \ldh, \fpsl$.
\begin{enumerate}
 \item \label{prop:SCH:cech} $\varinjlim_{\lambda \in \Lambda}\check{H}^n_\tau(T_\lambda, F) \stackrel{\sim}{\to} \check{H}^n_\tau(T, \iota^*F)$ for any presheaf $F \in \PreShv(\Sch_S)$.
 
 \item \label{prop:SCH:coh} $H^n_\tau(T, \iota^sF) = \varinjlim_{\lambda \in \Lambda} H^n_\tau(T_\lambda, F)$ for any sheaf $F \in \Shv_\tau(\Sch_S)$. 
 
 \item \label{prop:SCH:ldhRefine} Every $\ldh$-cover of $T$ is refinable by the composition of a $\cdh$-cover \mbox{$\{V_i \to T\}_{i = 1}^n$} and $\fpsl$ morphisms $W_i \to V_i$. 

 \item \label{prop:SCH:hvrrefine} Every $\ldh$-cover of the spectrum of an hvr %$R$ 
 is refinable by an $\fpsl$-morphism. %$\Spec(B) \to \Spec(R)$. If $R$ has finite rank, we can choose $B$ to be integral, and such that $B \to \widetilde{B}$ is a $\cduh$, and $\widetilde{B}$ is also a hvr.

 \item \label{prop:SCH:bu} If $Z \to T$ a closed immersion, and $Y \to T$ a proper morphism, such that $Y \setminus Z \times_T Y \to T \setminus Z$ is an isomorphism, then 
 \[ 0 \to \ZZ(Z \times_T Y) \to \ZZ(Z) \oplus \ZZ(Y) \to \ZZ(T) \to 0 \]
becomes a short exact sequence after $\cdh$-sheafification, where $\ZZ(W) := \ZZ \hom_{\SCH_S}(-, W)$.

 \item If $F \in \PreShv(\Sch_S)$ has a structure of traces, then $\iota^*F \in \PreShv(\SCH_S)$ inherits a canonical structure of traces extending that of $F$.
 
 \item \label{prop:SCH:uhInv} If $F \in \PreShv(\Sch_S)$ is $\uh$-invariant, then $\iota^*F  \in \PreShv(\SCH_S)$ is invariant for finitely presented $\uh$-morphisms, and $\uh$-morphisms between affine schemes with finitely many points.%For example, any scheme finite over a finite rank hvr, the integral closure, and pseudo-integral closure of any such scheme in a finite extension of its ring of total fractions 
\end{enumerate}
\end{prop}

\begin{rema}
In part \eqref{prop:SCH:uhInv} we can actually prove that $\iota^*F$ is invariant for all $\uh$-morphisms, assuming only that $S$ is quasi-compact and quasi-separated, but as we don't need this stronger more general statement in this present work, we do not include its proof. In fact, David Rydh explained to us that the proof below works more or less unchanged, with quasi-compactness of the constructible topology in place of the hypothesis that the schemes have finitely many points. 
\end{rema}

\begin{proof}\ 
\begin{enumerate}
 \item Both surjectivity and injectivity follows directly from Proposition~\ref{prop:inducedSCH} together with \cite[Thm.8.8.2]{EGAIV3} saying that morphisms over $T$ lift through the filtered system $\{T_\lambda\}$.

%Surjectivity: Suppose $a \in \check{H}^n_\tau(T, \iota^*F)$ is represented by some $a \in \check{H}^n(\mathcal{U}/T, \iota^*F)$ on some covering $\mathcal{U} \to T$. By Prop{.}\ref{prop:inducedSCH} we can assume that $\mathcal{U}$ is the pullback of some $\mathcal{U}_\lambda \to T_\lambda$. But $\check{H}^n(\mathcal{U}/T, \iota^*F) = \varinjlim_{\mu < \lambda} \check{H}^n(\mathcal{U}_\mu /T_\mu, F)$ where $\mathcal{U}_\mu = T_\mu {\times_{T_\lambda}} \mathcal{U}_\lambda$. So surjectivity follows. 
 
%Injectivity: Suppose $a \in \check{H}^n(\mathcal{U}_\lambda /T_\lambda, F)$ is sent to zero. Then there is a covering family $\mathcal{V}$ of $T$ refining $T \times_{T_\lambda} \mathcal{U_\lambda}$ on which the image of $a$ vanishes. We can assume that $\mathcal{V}$ is the pullback of a finite cover $\mathcal{V}_\mu$ from some $T_\mu$ and that $\mu < \lambda$, so replacing $\lambda$ with $\mu$, we assume $\mu = \lambda$. By \cite[Thm.8.8.2]{EGAiv3} we have $\hom_T(\mathcal{V}, T {\times_{T_\lambda}} \mathcal{U}) = \varinjlim_{\mu < \lambda} \hom_{T_\mu}(T_\mu {\times_{T_\lambda}} \mathcal{V}_\lambda, T_\mu {\times_{T_\lambda}} \mathcal{U}_\lambda)$, so, possibly refining $\lambda$ further, we can assume that the refinement $\mathcal{V} \to T \times_{T_\lambda} \mathcal{U}_\lambda$ is the pullback of a $T_\lambda$-morphism $\mathcal{V}_\lambda \to \mathcal{U}_\lambda$. Then since $\check{H}^n(\mathcal{V}/T, \iota^*F) = \varinjlim_{\mu < \lambda} \check{H}^n(\mathcal{V}_\mu /T_\mu, F)$, the section $a$ is already zero in the colimit on the left. 

 \item As $i^s$ is exact, the functor $i_s$ preserves injective resolutions.%
 \footnote{If the reader is worried $\SCH_S$ is too big for injective resolutions, then just choose some large enough regular cardinal $\kappa$ and instead work with the category $\SCH_S^{\leq \kappa}$ of quasi-compact separated $S$-schemes which are filtered limits of filtered systems in $\Sch_S$ with affine transition morphisms indexed by a category $\Lambda$ with $< \kappa$ morphisms. Then $\SCH_S^{\leq \kappa}$ will be essentially small.
 } %
Then any injective resolution of $\iota^sF \to \mathcal{I}^\bullet$ restricts to an injective resolution $F = \iota_s\iota^s F \to \iota_s\mathcal{I}^\bullet$, and we have $H^n_\tau(T, \iota^sF) 
= (H^n\mathcal{I}^\bullet)(T) 
% Yoneda commutes with limits so $h_T = \varprojlim h_{T_\lambda}$ on $\SCH_S$ (evaluate on a test scheme $T'$). Sheafification is exact. If $I$ is injective then $\hom(-, I)$ is exact (by definition).
= (H^n\mathcal{I}^\bullet)(\varprojlim T_\lambda) 
\stackrel{(*)}{=} \varinjlim (H^n\mathcal{I}^\bullet)(T_\lambda) 
= \varinjlim (H^n\iota_s\mathcal{I}^\bullet)(T_\lambda) 
= \varinjlim H^n_\tau(T_\lambda, F)$. For (*) note for any injective sheaf $I$, $\hom(-, I)$ is exact by definition, and Yoneda, and sheafification are both left exact.

 \item By Proposition~\ref{prop:inducedSCH}, this follows from the case where $T \in \Sch_S$ which is well-known, \cite[Prop.2.1.12(iii)]{Kel17}. This latter is proved using Raynaud-Gruson flatifiaction, cf. the proof of Lemma~\ref{lemm:uhldh}.
 
 \item By part \eqref{prop:SCH:ldhRefine}, every $\ldh$-cover is refinable by a $\fpsl'$-cover followed by a $\cdh$-cover. But every $\cdh$-cover of a hvr has a section: for completely decomposed proper morphisms this follows from the valuative criterion for properness, and for completely decomposed étale morphisms, this follows from Hensel's Lemma.
 
 \item To check exactness, it suffices to check exactness after evaluating the sequence of presheaves on an hvr. But in this case one readily checks exactness using the valuative criterion for properness.
 
 \item Given a finite flat surjective morphism $f: Y \to X$ in $\SCH_S$, there is a filtered system $(X_\lambda)$ in $\Sch_S$ with affine transition morphisms, Rem.\ref{rema:SCH}, such that $X = \varprojlim X_\lambda$, there is some $\lambda$ and $f_\alpha: Y_\alpha \to X_\alpha$ in $\Sch_S$ such that $f = X \times_{X_\alpha} f_\alpha$, \cite[Thm.8.8.2(ii)]{EGAIV3}. We can assume $f_\alpha$ is surjective and finite, \cite[Thm.8.10.5]{EGAIV3}. By restricting to finitely many affine opens $U \subseteq X_\alpha$, and choosing isomorphism inducing global sections $\OO_X^d \to f_*\OO_Y$, we see that we can also assume $f_\alpha$ is flat. Note that noetherian-ness was used here to kill $\ker(\OO_X^d \to f_* \OO_Y)$.
 
Now define trace morphisms on $\iota^*F$ by choosing such presentations and define $\Tr_f$ to be $\varinjlim \Tr_{f_\lambda}$. Note that this is well defined. Let $a \in F(Y)$ be represented by some $a_\lambda \in F(X_\lambda \times_{X_\alpha} Y_\alpha)$, and let $(X'_\lambda), f'_\alpha: Y'_\alpha \to X'_\alpha, a'_{\lambda'} \in F(X'_\lambda \times_{X'_{\alpha'}} Y'_{\alpha'})$ be some other choice of representative. By \cite[Prop.8.13.1]{EGAIV3} the canonical morphism $X \to X'_{\lambda'}$ factors as some $X \to X_\mu \to X'_{\lambda'}$. Since $X {\times_{X_\alpha}} Y_\alpha \cong Y \cong X {\times_{X'_{\alpha'}}} Y'_{\alpha'}$, there exists some possibly smaller $\mu$ and an isomorphism $X_\mu \times_{X_\alpha} Y_\alpha \cong X_\mu {\times_{X'_{\alpha'}}} Y'_{\alpha'}$ compatible with the former \cite[Thm.8.8.2(i)]{EGAIV3}. As $a_\lambda$ and $a'_{\lambda'}$ agree in the colimit $F(Y)$, possibly making $\mu$ smaller again, we can assume $a_\lambda$ and $a'_{\lambda'}$ already agree in $F(X_\mu \times_{X_\alpha} Y_\alpha) \cong F(X_\mu {\times_{X'_{\alpha'}}} Y'_{\alpha'})$. Then it follows from (CdB) that $\Tr_{f_\lambda}(a_\lambda)$ agrees with $\Tr_{f'_{\lambda'}}(a'_{\lambda'})$ in $F(X_\mu)$, and therefore also in the colimit $F(X)$. 

(Add) Given $f: Y \to X$ and $f': Y' \to X'$, choose representatives for $Y \to X \sqcup X'$ and $Y \to X \sqcup X'$ separately. Then (Add) follows.

(Fon) Given $g: W \to Y$ and $f: Y \to X$ choose a representative for $f$, and then use the system $\{X_\lambda \times_{X_\alpha} Y_\alpha\}$ to choose a representative for $g$. Then (Fon) follows.

(CdB) Choose a representative for $f$, then descend $W$ to $(X_\lambda)$ using \cite[Thm.8.8.2(ii)]{EGAIV3}.

(Deg) is clear.

 \item If $f: T' \to T$ is a finitely presented universal homeomorphism in $\SCH_S$, then there exists some $\alpha \in \Lambda$, and a universal homeomorphism $f_\alpha: T'_\alpha \to T_\alpha$ in $\Sch_S$, such that $f = T \times_{T_\alpha} f_\alpha$, \cite[Thm.8.8.2(ii),Thm.8.10.5(vi)(vii)(viii)]{EGAIV3}, so $\iota^*F(f) = \varinjlim_{\lambda \leq \alpha} F(T \times_{T_\alpha} f_\alpha)$ is an isomorphism. Note that EGA’s ``radiciel'' is equivalent to universally injective.
 
 Suppose that $f: T' = \Spec(\OO_{T'}) \to \Spec(\OO_{T}) = T$ is a universal homeomorphism between affine schemes with finitely many points. Write $\OO_{T'} = \varinjlim_{\OO_T \to A \stackrel{\phi}{\to} \OO_{T'}} A$ as a filtered colimit of finitely presented $\OO_T$-algebras.
 
As $T' \to T$ is a universal homeomorphism each $\Spec(\phi(A)) \to T$ is a finite universal homeomorphism, Lem.\ref{lemm:subuh}. Also, $\phi(A) = \varinjlim_{I \subseteq \ker(A \to \phi(A))} A / I$ is the filtered colimit of the quotients of $A$ by the finitely generated ideals of the kernel $K = \ker(A \to \phi(A))$. Since $\Spec(A)$ has finitely many points, we can always find some finitely generated ideal $J$ such that for each $J \subseteq I \subseteq K \subseteq A$ the closed immersion $\Spec(\phi(A)) \to \Spec(A / J)$ is surjective. Hence, $f: T' \to T$ is a filtered limit of universal homeomorphisms of finite presentation. We have already seen that $\iota^*F$ sends such morphisms to isomorphisms, so $\iota^*F(f)  = \varinjlim F(\Spec(A) \to T)$ is an isomorphism. 
\end{enumerate}
\end{proof}

%\begin{center}
%\includegraphics[width=\textwidth]{ShujiSong.pdf}
%\end{center}

\bibliographystyle{abstract}
\bibliography{bib}

\begin{thebibliography}{CTHK97}

\bibitem[BM18]{BM18}
Bhargav Bhatt and Akhil Mathew.
\newblock The arc-topology, 2018.

\bibitem[Bou64]{Bou64}
N.~Bourbaki.
\newblock {\em \'{E}l\'ements de math\'ematique. {F}asc. {XXX}. {A}lg\`ebre
  commutative. {C}hapitre 5: {E}ntiers. {C}hapitre 6: {V}aluations}.
\newblock Actualit\'es Scientifiques et Industrielles, No. 1308. Hermann,
  Paris, 1964.

\bibitem[CD15]{CD15}
Denis-Charles Cisinski and Fr{\'e}d{\'e}ric D{\'e}glise.
\newblock Integral mixed motives in equal characteristic.
\newblock {\em Documenta Math. Extra Volume: Alexander S. Merkurjev’s
  Sixtieth Birthday}, pages 145--194, 2015.

\bibitem[CTHK]{CTHK}
Jean-Louis Colliot-Th{\'e}lene, Raymond~T Hoobler, and Bruno Kahn.
\newblock The {B}loch-{O}gus-{G}abber theorem.
\newblock {\em Algebraic K-theory (Toronto, ON, 1996)}, 16:31--94, 1997.

\bibitem[EGAII]{EGAII}
A.~Grothendieck.
\newblock \'{E}l\'ements de g\'eom\'etrie alg\'ebrique. {II}. \'{E}tude globale
  \'el\'ementaire de quelques classes de morphismes.
\newblock {\em Inst. Hautes \'Etudes Sci. Publ. Math.}, (8):222, 1961.

\bibitem[EGAIV2]{EGAIV2}
A.~Grothendieck.
\newblock \'{E}l\'ements de g\'eom\'etrie alg\'ebrique. {IV}. \'{E}tude locale
  des sch\'emas et des morphismes de sch\'emas. {II}.
\newblock {\em Inst. Hautes \'Etudes Sci. Publ. Math.}, (24):231, 1965.

\bibitem[EGAIV3]{EGAIV3}
A.~Grothendieck.
\newblock \'{E}l\'ements de g\'eom\'etrie alg\'ebrique. {IV}. \'{E}tude locale
  des sch\'emas et des morphismes de sch\'emas, troisi{\`e}me partie.
\newblock {\em Inst. Hautes \'Etudes Sci. Publ. Math.}, (28):255, 1966.
\newblock Revised in collaboration with Jean Dieudonn{\'e}. Freely available on
  the Numdam web site at
  \url{http://www.numdam.org/numdam-bin/feuilleter?id=PMIHES_1966__28_}.

\bibitem[EGAIV4]{EGAIV4}
Alexandre Grothendieck.
\newblock \'{E}léments de géométrie algébrique. {IV}. \'{E}tude locale des
  schémas et des morphismes de schémas {IV}.
\newblock {\em Inst. Hautes Études Sci. Publ. Math.}, (32):361, 1967.
\newblock Revised in collaboration with Jean Dieudonné. Freely available on
  the Numdam web site at
  \url{http://www.numdam.org/numdam-bin/feuilleter?id=PMIHES_1967__32_}.

\bibitem[EK18]{EK18}
Elden Elmanto and Adeel~A. Khan.
\newblock Perfection in motivic homotopy theory, 2018.

\bibitem[EP05]{EP05}
Antonio~J Engler and Alexander Prestel.
\newblock {\em Valued fields}.
\newblock Springer Science \& Business Media, 2005.

\bibitem[GK15]{GK15}
Ofer Gabber and Shane Kelly.
\newblock Points in algebraic geometry.
\newblock {\em Journal of Pure and Applied Algebra}, 219(10):4667--4680, 2015.

\bibitem[HK18]{HK18}
Annette Huber and Shane Kelly.
\newblock Differential forms in positive characteristic {II}: cdh-descent via
  functorial {R}iemann-{Z}ariski spaces.
\newblock {\em {A}lgebra and {N}umber {T}heory}, Forthcoming, 2018.
\newblock Preprint \href{http://arxiv.org/abs/1706.05244}{arXiv:1706.05244}.

\bibitem[KM18]{KM18}
Shane Kelly and Matthew Morrow.
\newblock $k$-theory of valuation rings, 2018.

\bibitem[KS06]{KS06}
Masaki Kashiwara and Pierre Schapira.
\newblock {\em Categories and sheaves}, volume 332 of {\em Grundlehren der
  Mathematischen Wissenschaften [Fundamental Principles of Mathematical
  Sciences]}.
\newblock Springer-Verlag, Berlin, 2006.

\bibitem[Kel12]{Kel12}
Shane Kelly.
\newblock {\em Triangulated categories of motives in positive characteristic}.
\newblock PhD thesis, Universit{\'e} Paris 13, Australian National University,
  2012.
\newblock Preprint \href{http://arxiv.org/abs/1305.5349}{arXiv:1305.5349}.

\bibitem[Kel14]{Kel14}
Shane Kelly.
\newblock Vanishing of negative {$K$}-theory in positive characteristic.
\newblock {\em Compos. Math.}, 150(8):1425--1434, 2014.

\bibitem[Kel17]{Kel17}
Shane Kelly.
\newblock Voevodsky motives and $l$dh-descent.
\newblock {\em Ast\'erisque}, (391):iv+125, 2017.

\bibitem[Mat89]{Mat89}
Hideyuki Matsumura.
\newblock {\em Commutative ring theory}, volume~8 of {\em Cambridge Studies in
  Advanced Mathematics}.
\newblock Cambridge University Press, Cambridge, second edition, 1989.
\newblock Translated from the Japanese by M. Reid.

\bibitem[Qui73]{Qui73}
Daniel Quillen.
\newblock {\em Higher algebraic K-theory: I}, pages 85--147.
\newblock Springer Berlin Heidelberg, Berlin, Heidelberg, 1973.

\bibitem[RG71]{RG71}
Michel Raynaud and Laurent Gruson.
\newblock Crit\`eres de platitude et de projectivit\'e. {T}echniques de
  ``platification'' d'un module.
\newblock {\em Invent. Math.}, 13:1--89, 1971.

\bibitem[Ryd10]{Ryd10}
David Rydh.
\newblock Submersions and effective descent of {\'e}tale morphisms.
\newblock {\em Bull. Soc. Math. France}, 138:181--230, 2010.

\bibitem[SGA4]{SGA4}
{\em Th\'eorie des topos et cohomologie \'etale des sch\'emas. {T}ome 1:
  {T}h\'eorie des topos}.
\newblock Lecture Notes in Mathematics, Vol. 269. Springer-Verlag, Berlin,
  1972.
\newblock S{\'e}minaire de G{\'e}om{\'e}trie Alg{\'e}brique du Bois-Marie
  1963--1964 (SGA 4), Dirig{\'e} par M. Artin, A. Grothendieck, et J. L.
  Verdier. Avec la collaboration de N. Bourbaki, P. Deligne et B. Saint-Donat.

\bibitem[SGA42]{SGA42}
{\em Th\'eorie des topos et cohomologie \'etale des sch\'emas. {T}ome 2}.
\newblock Lecture Notes in Mathematics, Vol. 270. Springer-Verlag, Berlin,
  1972.
\newblock S{\'e}minaire de G{\'e}om{\'e}trie Alg{\'e}brique du Bois-Marie
  1963--1964 (SGA 4), Dirig{\'e} par M. Artin, A. Grothendieck et J. L.
  Verdier. Avec la collaboration de N. Bourbaki, P. Deligne et B. Saint-Donat.

\bibitem[Stacks]{Stacks}
The {Stacks Project Authors}.
\newblock {\itshape Stacks Project}.
\newblock {\url{http://stacks.math.columbia.edu}}, 2014.

\bibitem[Tem11]{Tem11}
Michael Temkin.
\newblock Relative {R}iemann-{Z}ariski spaces.
\newblock {\em Israel Journal of Mathematics}, 185(1):1--42, 2011.

\bibitem[Voe00]{Voe00}
Vladimir Voevodsky.
\newblock Triangulated categories of motives over a field.
\newblock In {\em Cycles, transfers, and motivic homology theories}, volume 143
  of {\em Ann. of Math. Stud.}, pages 188--238. Princeton Univ. Press,
  Princeton, NJ, 2000.

\bibitem[Voe00b]{Voe00b}
Vladimir Voevodsky.
\newblock Cohomological theory of presheaves with transfers.
\newblock In {\em Cycles, transfers, and motivic homology theories}, volume 143
  of {\em Ann. of Math. Stud.}, pages 87--137. Princeton Univ. Press,
  Princeton, NJ, 2000.

\bibitem[Voe96]{Voe96}
Vladimir Voevodsky.
\newblock Homology of schemes.
\newblock {\em Selecta Math. (N.S.)}, 2(1):111--153, 1996.

\end{thebibliography}

\end{document}